%% file: AlgomKfir10.tex
\documentclass[10pt]{amsart}
\usepackage{amsfonts}
\usepackage{mathrsfs} 
\usepackage{amsmath}
\usepackage{amssymb}
\usepackage{hyperref}
\usepackage{xypic}
\usepackage{amsthm,amssymb,epsfig,graphicx,amscd,amsmath,color}

\theoremstyle{plain}
\newtheorem{theorem}{Theorem}[section]
\newtheorem*{thm}{Theorem}
\newtheorem*{ML}{The Morse Lemma}

\newtheorem{cor}[theorem]{Corollary}
\newtheorem*{cor1}{Corollary}
\newtheorem{lemma}[theorem]{Lemma}
\newtheorem*{claim}{Claim}

\newtheorem{prop}[theorem]{Proposition}
\newtheorem{definition-lemma}[theorem]{Definition-Lemma}

\theoremstyle{definition}
\newtheorem{definition}[theorem]{Definition}
\newtheorem{defn}[theorem]{Definition}
\newtheorem{example}[theorem]{Example}
\newtheorem{remark}[theorem]{Remark}

\newtheorem{obs}[theorem]{Observation}

\newtheorem{notation}[theorem]{Notation}

\newcommand {\RR}{\mathbb{R}}

\newcommand {\RP}{\mathbb{RP}}
\newcommand {\ZZ}{\mathbb{Z}}
\newcommand {\NN}{\mathbb{N}}
\newcommand {\LL}{\mathcal{L}}

\newcommand {\into}{\hookrightarrow}

\newcommand {\BB}{\mathcal{B}}
\newcommand {\Q}{\mathcal{Q}}

\newcommand {\inn}{\subseteq}

\newcommand {\al}{\alpha}
\newcommand {\out}{\mbox{Out}(F_n)}

\newcommand {\wt}{\widetilde}
\newcommand {\eps}{\epsilon}
\newcommand {\Lam}{\Lambda}
\newcommand {\lam}{\lambda}
\newcommand {\sig}{\sigma}
\newcommand {\ws}{\wt\sig}

\newcommand {\ST}{\textup{St}}
\newcommand {\im}{\textup{Im}}
\newcommand {\diam}{\textup{diam}}
\newcommand {\Lip}{\textup{Lip}}

\newcommand {\C}{\mathcal{C}}
\newcommand {\cayout}{\mathscr{C}}

\newcommand {\id}{\textup{id}}
\newcommand {\ve}{\varepsilon}
\newcommand {\tr}{\textup{tr}}
\newcommand {\lip}{\textup{Lip}}
\newcommand {\mcg}{\mathcal{MCG}}
\newcommand {\teich}{Teichm\"uller }
\newcommand {\Cone}{\text{Cone}}
\newcommand {\rb}[2]{B_{#2}(#1_{\rightarrow} ) }
\newcommand {\lb}[2]{B_{#2}(#1_{\leftarrow} ) }

\newcommand {\U}{\mathcal{U}}

\newcommand{\Teich}{\mathcal{T}(S_{g,p})}

\newcommand{\X}{\mathcal{X}}

\newcommand{\ra}{\rightarrow}

\newcommand{\os}{\mathcal{X}_n} 

\title{Strongly Contracting Geodesics in Outer Space}
\author{Yael Algom-Kfir}

\begin{document}
\begin{abstract}
We study the Lipschitz metric on Outer Space and prove that fully
irreducible elements of $\out$ act by hyperbolic isometries with axes
which are strongly contracting. As a corollary, we prove that the
axes of fully irreducible automorphisms in the Cayley graph of $\out$
are Morse, meaning that a quasi-geodesic with endpoints on the axis stays within a bounded distance from the axis. 
\end{abstract}

\maketitle

\section*{Introduction}

There exists a striking analogy between the mapping class groups of
surfaces, and the outer automorphism group $\out$ of a rank $n$ free
group. At the core of this analogy lies Culler and Vogtmann's
Outer Space $\os$ \cite{CV}, a contractible finite dimensional cell complex on which $\out$ has a
properly discontinuous action. Like \teich space, Outer Space has an
invariant spine on which the action is cocompact, making it a good
topological model for the study of $\out$. Indeed, Outer Space has
played a key role in proving theorems for $\out$, which were classically known
for the mapping class group. For example, the
action of a fully irreducible outer automorphism on the boundary of
$\os$ has been shown \cite{LL} to have North-South dynamics, and the
Tits alternative holds for $\out$ \cite{BFH00}, \cite{BFH05}. \\

However, while there have been several well studied metrics on \teich space 
(the \teich metric, the Weil-Petersson metric, and the Lipschitz
metric), the geometry of Outer Space has remained largely
uninvestigated (exceptions include \cite{HM} and \cite{FM}). One
would like to define a metric on Outer Space so that fully irreducible
elements of $\out$ (which are analogous to pseudo-Anosov elements in
$\mcg(S)$) act by hyperbolic isometries with meaningful translation
lengths. But immediately one encounters a problem: it isn't clear
whether to require the metric to be symmetric. To clarify, we follow
the discussion in \cite{HMpara}. Consider the situation of a
pseudo-Anosov map $\psi$ acting on \teich space $\Teich$ with the \teich metric
$d_{\text{T}}$. Associated to $\psi$ is an expansion factor
$\lam_\psi$ and two foliations $\mathcal{F}^s$ and $\mathcal{F}^u$ so
that $\psi$ expands the leaves of $\mathcal{F}^s$ by $\lam_{\psi}$
and contracts the leaves of $\mathcal{F}^u$ by $\lam_\psi^{-1}$.
Incidentally, $\lam_\psi= \lam_{\psi^{-1}}$. Furthermore, by \teich's theorem, the
translation length of $\psi$ is $\log(\lam_\psi)$. Going back to
$\out$, one can associate to a fully irreducible outer automorphism
$\Phi$ a Perron-Frobenius eigenvalue $\lam_\Phi$ which plays
much the same roll as the expansion factor in the pseudo-Anosov case.
 If we did have an honest metric on Outer Space where
$\Phi$ was a hyperbolic isometry then the axis for $\Phi$ would also
be an axis for $\Phi^{-1}$. Thus for a point $x$ on the axis of
$\Phi$, $d(x,\Phi(x)) = \log(\lam_\Phi)$ and $ d(\Phi(x),x) =
\log(\lam_{\Phi^{-1}})$. However, it is not always the case that $\lam_\Phi = \lam_{\Phi^{-1}}$. Therefore one would have to abandon either
the symmetry of the metric or the relationship between the
translation length $\Phi$ and its Perron-Frobenius 
eigenvalue. We choose to do the former in order to preserve the ties
between the action of $\Phi$ on $\os$ and its action on the
conjugacy classes of $F_n$. \\

The (non-symmetric) metric  that we carry over from $\Teich$ to Outer
Space is the Lipschitz metric introduced by Thurston \cite{T}. Given two marked hyperbolic structures $(X,f),(Y,g)$ on a surface $S$ 
define
\[ d_L((X,f), (Y,g)) = \inf \{  \Lip(h)  |  h \text{ is Lipschitz, and 
  homotopic to } g \circ f^{-1}  \} \]
In \cite{ChR} Rafi and Choi  proved that this metric is Lipschitz equivalent to $d_T$
in the thick part of $\Teich$. \\

While $\Teich$ with $d_{\text{T}}$ is not CAT(0) \cite{Ma} or
Gromov hyperbolic \cite{MaW} it does exhibit some features of negative curvature
in the thick part. A geodesic is \emph{strongly contracting} if its 
nearest point projection takes balls disjoint from the geodesic
to sets of bounded diameter, where the bound is independent of the
radius of the ball. That is, the ``shadow" that a ball casts on the
geodesic is bounded. For example, geodesics in a Gromov hyperbolic
space are strongly contracting. In \cite{Mi} Minsky proved that
geodesics contained in the $\eps$-thick part of $\Teich$ are
uniformly strongly contracting, with the bound only depending on $\eps$ and the topology of $S$. Note that any axis of a pseudo-Anosov map is contained in the $\eps$-thick part of $\Teich$ for a sufficiently small $\eps$. We prove

\begin{thm}
An axis of a fully irreducible outer automorphism is strongly contracting.
\end{thm}

A geodesic $L$ in a metric space is \emph{Morse} if every
quasi-geodesic segment with endpoints on $L$ stays within a bounded
neighborhood of $L$  which only depends on the
quasi-geodesic constants. As an application of the theorem above we
prove:

\begin{cor1}
In the Cayley graph of $\out$, the axis of a fully irreducible
automorphism is Morse.
\end{cor1}

This paper is organized as follows
\begin{itemize}
 \item In chapter 1 we go over some definitions and background on
 Outer Space. The well informed reader could skip this part.
  \item In Chapter 2 we define the
Lipschitz metric on Outer Space, and deduce a formula which
expresses the relationship between the metric and the lengths of
loops in $X$ and $Y$ (proof due to Tad White and first written in
\cite{FM}).  
\item In Chapter 3 we
describe axes of fully irreducible automorphisms. Given such an axis,
we define a coarse projection of $\os$ onto this axis. Note that the axis 
for $\Phi$ will not necessarily be an axis for $\Phi^{-1}$, however
the projections of a point to both axes are uniformly close.
  \item In Chapter 4 we define the Whitehead graph $Wh_X(\Lam^\pm)$ of the
attracting and repelling laminations of $\Phi$ at the point $X \in
\os $. We prove that there exists a point $F \in \os$ for which
$Wh_F(\Lam^+) \cup Wh_F(\Lam^-)$ is connected and does not contain a
cut vertex.
  \item In Chapter 5 we use the previous result to show
that any loop $\al$ which represents a basis element cannot contain long pieces of both laminations. Next we prove our
main ``negative curvature" property. If the projections of $x$ and $y$
are sufficiently far apart then then $d(x,y)$ is coarsely larger than
$d(x,p(x)) + d(p(x),p(y))$. We show that this is enough to prove that $L$ is a strongly contracting geodesic. We end the Chapter by proving that in the Cayley graph of $\out$, axes of fully irreducible automorphisms are Morse.
  \item In Chapter 6 we have
collected some applications: the asymptotic cone of the Cayley graph of $\out$ contains
many cut points and is in fact tree graded, the divergence function
in $\os$ is at least quadratic. Finally, we show that projections onto two
axes $A,B$ of independent irreducible automorphisms satisfy a dichotomy
similar to the one shown in \cite{Beh} for subsurface projections. 

\end{itemize}

\emph{A note on notation}:
Many of the theorems and propositions in this article contain several
constants which we usually denote s or c within the proposition. When
referring to a constant from a previous proposition, we add its 
number as a subscript.

\subsection*{Acknowledgments} It is a pleasure to thank my advisor
Mladen Bestvina for investing many hours of his time and a few of his 
brilliant ideas in this work. I would also like to thank Mark Feighn
for his support. Finally, I'd like to thank Jason Behrstock and Ken Bromberg for their help with chapter 6.


\section{Preliminary notions}

\subsection*{The Outer Automorphism Group of the Free Group}
\begin{defn}
The group of outer automorphisms of the free group of rank $n$ is
$$ \out = \textup{Aut}(F_n)/ \textup{Inner}(F_n)$$
If $\phi$ is an automorphism we denote its class by $\Phi$. 
\end{defn}

\begin{defn}\label{Wh_move}
Let $\{ x_1, \dots , x_n \}$ be a basis for $F_n$. Let $A \subset \{ x_1^{\pm 1}, \dots , x_n^{\pm 1} \}$ and $a \in A$ so that $a^{-1} \notin A$. Then the {\it Whitehead automorphism} $\phi_{(A,a)}$ associated with $(A,a)$ is defined as follows.  $\phi_{(A,a)}(a) =a$ and for $x \neq a, a^{-1}$: 
\[ \begin{array}{lll}
	x \to a x a^{-1} &	& \text{if } x, x^{-1} \in A \\
	x \to xa^{-1} &		& \text{if } x \in A \text{ and } x^{-1} \notin A \\
	x \to ax	&		& \text{if } x \notin A \text{ and } x^{-1} \in A \\
	x \to x	&		& \text{if } x,x^{-1} \notin A
\end{array}\]
\end{defn}

\begin{thm}[Whitehead Generators] 
The following set generates $\out$ $$\{ [\phi_{(A,a)}] \mid \text{ all possible } a,A \}$$ 
\end{thm}

\subsection*{Bases of $\out$ and Whitehead's Theorem}

\begin{defn}[The Whitehead Graph]
Let $\mathcal{B} = \{ y_1, \dots , y_n \}$ be a basis of $F_n$, let $\mathfrak{a}$ be the conjugacy class, in $F_n$ and $w \in \mathfrak{a}$ a cyclically reduced word written in the basis $\mathcal{B}$. Then the
{\it Whitehead graph} of $\mathfrak{a}$ with respect to $\mathcal{B}$ is denoted $Wh_{\mathcal{B}}(\mathfrak{a})$ and
constructed as follows: The vertex set of this graph is the set
$\mathcal{B} \cup \mathcal{B}^{-1}$. $z_i$ and $z_j$ are connected by an edge if $z_i^{-1} z_j$ or $z_j^{-1} z_i$ appears in the cyclic word $w$, i.e. if $w = \dots z_i^{-1} z_j \dots$ or $w = \dots z_j^{-1} z_i \dots $ or $w = z_j \dots z_i^{-1}$ or $w = z_i \dots z_j^{-1}$. \\
The Whitehead graph $Wh_\mathcal{B}([w_1], \dots , [w_k] )$ of the set $\{ [w_1], \dots , [w_k] \}$ is the superposition of the individual Whitehead graphs $Wh_\mathcal{B}([w_i])$. 
\end{defn}

\begin{defn}
We say that the set of conjugacy classes $\{ \mathfrak{a}_1 , \dots , \mathfrak{a}_k \}$ \emph{can be completed to a basis} if there are $w_i \in \mathfrak{a}_i$ and $w_{k+1}, \dots, w_n$ such that $\{ w_1, \dots , w_n \}$  is a basis for $F_n$.
\end{defn}

A cut vertex in a graph is a vertex that when removed, leaves the graph disconnected.

\begin{theorem}[Whitehead \cite{W}]\label{WhThm}
If $\mathfrak{a}_1, \dots, \mathfrak{a}_k$ can be completed to a basis and   $Wh_\mathcal{B}(\mathfrak{a}_1, \dots, \mathfrak{a}_k)$ is connected, then $Wh_\mathcal{B}( \mathfrak{a}_1, \dots, \mathfrak{a}_k)$ has a cut vertex.
\end{theorem}

A related notion is the following

\begin{defn}[Free Factor]
\label{free_factors}Let $A$ be a subgroup of $F_n$, $A$ is a free factor of $F_n$ if there exists a subgroup $B$ such that $F_n = A*B$.
\end{defn}

Let $\BB = \{ x_1, \dots x_n \}$ be a basis for $F_n$, and suppose we want to determine if $[y_1], \dots , [y_k]$ can be completed to a basis, where $y_i$ are cyclically reduced words. We construct $W = Wh_{\BB}([y_1], \dots ,
[y_k])$. 
\begin{itemize}
	\item If $W$ is connected and does not contain a cut vertex then $[y_1], \dots [y_k]$ cannot be completed to a basis by Theorem \ref{WhThm}.
	\item If $W$ is connected and contains a cut vertex $a$ we will construct a new basis $\BB'$ such that 
	\begin{equation}\label{decrease}
	\sum^{k}_{i=1}|y_i|_{\mathcal{B}} > \sum^{k}_{i=1}|y_i|_{\mathcal{B}'}
	\end{equation}
	Let $W'$ be the induced subgraph of $W$ on all of the vertices except $a$.  Let $W''$ be a connected component of $W'$, which does not contain $a^{-1}$. Take $A$ to be the elements of $\BB$ whose vertices are in $W''$ and $a$. Then $\BB' = \phi_{(A,a)}(\BB)$. It is straightforward to check that (\ref{decrease}) is satisfied. 
	\item  If $W$ is not connected, one can continue carrying out the algorithm on a subset of the generators, until either there is some subset of the $y_i$s that do not form a free factor (since their Whitehead graph is connected with no cut vertex) or there is a basis $\BB''$ where $|y_i|_{\BB''} = 1$ for all $i$, which means $\BB''$ contains $y_1, \dots y_k$.
\end{itemize}

Whitehaed's theorem \ref{WhThm} may be reformulated in the following way. 

\begin{definition}
$[y_1], [y_2], \dots , [y_k]$ are compatible with a free decomposition of $F_n$, if there exists a free splitting $A*B$ so that for all $i$, either $[y_i] \in [A]$ or $[y_i] \in [B]$.
\end{definition}

\begin{theorem}[\cite{M}]\label{WhVer2} The following are equivalent:
\begin{enumerate}
    \item $[y_1], \dots , [y_k]$ are compatible with a free decomposition of
      $F_n$.
    \item If $\BB$ is a basis such that $Wh_\BB([y_1], \dots ,[y_k])$
    contains no cut vertex then it is disconnected.
    \end{enumerate}
\end{theorem}

\subsection*{Outer Space}
A {\it graph} will always
be a finite cell complex of dimension 1 with all vertices of valence
$>2$. A {\it metric} on a graph $G$ is a function $\ell:E(G)\to [0,1]$
defined on the set of edges of $G$ such that
\begin{itemize}
\item $\sum_{e\in E(G)}\ell(e)=1$. We shall denote the total sum of lengths of edges in the metric graph $G$ by $vol(G)$.
\item $\cup_{\ell(e)=0}e$ is a forest, i.e. it contains no circles.
\end{itemize}
The space $\Sigma_G$ of all such metrics $\ell$ on $G$ is a
``simplex with missing faces''; the missing faces correspond to
degenerate metrics that vanish on a subgraph which is not a forest. If $G'$ is obtained from $G$ by collapsing a forest, then we
have a natural inclusion $\Sigma_{G'}\subset \Sigma_G$. \\
 \indent The rose $R_0$ is the wedge of $n$ circles. A {\it marking} is a
homotopy equivalence $f: R_0 \to G$ from the rose to a graph. A {\it
  marked graph} is a pair $(G,f)$ where $f:R_0 \to G$ is a
marking. 

\begin{defn}[Outer Space - Graph Definition]
Outer Space $\os$ consists of the set of equivalence classes of triples $(G,f,\ell)$ where $G$ is a graph, $\ell$ is a metric, and $f$ is a marking, and so that $(G, f, \ell) \sim (G',f', \ell')$ if there is an isometry $\phi:(G,\ell) \to (G',\ell')$ so that $\phi \circ f$ is homotopic to $f'$.
\end{defn}

\begin{defn}
Throughout the paper we will abuse notation by referring to a point in $\os$ as $G$. 
\end{defn}

An equivalent definition is 
\begin{defn}(Outer Space - Tree Definition)
Outer Space $\os$ is the space of 
equivalence classes of free, simplicial, minimal $F_n$-trees, with the equivalence: $T \sim T'$ if there exists an $F_n$-equivariant homothety $\rho:T \to T'$.
\end{defn}

There is a natural right action of $\text{Aut}(F_n)$, the group of automorphisms of $F_n$, on $\os$. Let $\phi$ be an automorphism and let $g: R_0 \to R_0$ be a map such that $g_*: \pi_1(R_0, \text{ver}) \to \pi_1(R_0, \text{ver})$ equals $\phi$. Then $$\begin{array}{rrcl}
\phi: &\os & \to & \os \\[0.2 cm]
& (X,f,\ell) & \to & (X, f \circ g, \ell)
\end{array}$$ Notice that this action does not depend on the choice of $g$, and that inner automorphisms act trivially. Thus we get an action of $\out$ on $\os$. When we define the Lipschitz metric it will be evident that this is an isometric action.

\subsection*{The axes topology} Consider the set of non-trivial
conjugacy classes $\mathcal{C}$ in $F_n$. Each $F_n$-tree $T$ induces
a length function $\ell_T: F_n \to \RR$ by $\ell_T(x) =\tr(x)$ the
translation length of $x$ as an isometry of $T$. Since the
translation length is a class function, $\ell_T$
descends to a map $\ell_T: \mathcal{C} \to \RR$. Therefore we can
define a map \[
\begin{array}{rcc}
    \ell: \os & \to & \RP^\mathcal{C} \\
     \phantom{.} [ T ] & \to & [ \ell_T ]
\end{array} \] In \cite{CM} Culler and Morgan proved that this map is
injective. Thus $\os$ inherits a topology from $\RP^\mathcal{C}$
known as the axes topology. We remark (although we will not
need this) that there are other ways to define a topology on $\os$:
using the cellular structure of $\os$, and using the Gromov topology
on the space of metric $F_n$-trees. Paulin \cite{P} proved that all
three topologies are equivalent.

\subsection*{The boundary of Outer Space}
In \cite{CM} Culler and Morgan showed that $\overline{\os}$ is
compact. It was later shown in \cite{CohL} and \cite{BF} that
$\overline{\os}$ is the space of homothety classes of very small, minimal $F_n$-trees.

\subsection*{Train-track structures and maps} 
\begin{defn}[Turns and Train-Track Structures]
Let $G$ be a graph. An unordered pair of oriented edges $\{ e_1,e_2 \}$ is a \emph{turn} if $e_1, e_2$ have the same initial endpoint. Let $\bar{e}$ denote the edge $e$ with the opposite orientation. If an edge path $\al = \cdots \overline{e_1} e_2 \cdots $ or $\al = \cdots \overline{e_2} e_1
\cdots $ then we say that $\al$ \emph{crosses} or \emph{contains} the turn $\{ e_1, e_2 \}$. \\
A \emph{train track structure} on $G$ is an equivalence relation on
the set of oriented edges $E(G)$ with the property that if $e_1 \sim
e_2$ then $\{ e_1, e_2 \}$ is a turn.
\end{defn}

\begin{defn}[Legal Turns, and Gates]
A turn $\{ e_1,e_2 \}$ is \emph{legal} with respect to a fixed
train-track structure on $G$ if $e_1 \nsim e_2$. An edge path is
legal if every turn it crosses is legal. The equivalence classes of
the edges are called \emph{gates}. 
\end{defn}

\begin{defn}[A t-t structure induced by a self-map]\label{self-induced tt structures}
Let $g:G \to G$ be map which restricts on each edge to either an immersion or a constant map. The train-track structure induced by $g$ is the following equivalence relation: $e_1 \sim e_2$ if they have the same initial endpoint and there is some $m \geq 1$ such that there are small enough initial subsegments of $g^m(e_1)$ and $g^m(e_2)$ which coincide. 
\end{defn}

\begin{defn}[Train-track maps]
Let $g:G \to G$ be map which restricts on each edge to either an immersion or a constant map. $g$ is \emph{a (weak) train-track map} if for all $e \in E(G)$, the path $g(e)$ is legal (with respect to the t-t structure in \ref{self-induced tt structures}). A weak train-track map is called a \emph{train-track map} if in addition, vertices are mapped to vertices. 
\end{defn}
For us, the distinction between a train-track map and a weak train-track map will not be important so we subsequently drop the adjective ``weak".

\subsection*{Irreducible Outer Automorphisms}

\begin{defn}
Let $\Phi$ be an outer automorphism. $\Phi$ is \emph{reducible} if there exists a free product decomposition    \[ F_n=H_1* \dots * H_m* U \]
where all $H_i$ are nontrivial, $m \geq 1$ and where $\Phi$ permutes the conjugacy classes of $H_1,...,H_m \subseteq F_n$. An outer automorphism $\Phi$   is said to be \emph{irreducible} if it is not reducible. 
\end{defn}

\begin{defn}
$\Phi$ is called \emph{fully irreducible} if all of its powers are irreducible. 
\end{defn}

The main content of the classification theorem of $\out$ is that irreducible outer automorphisms have ``nice" representatives that are called train-track maps. 

\begin{defn}\label{topo_repn}
Let $\Phi \in \out$ \emph{a topological representative} of $\Phi$ is a marked graph $(G,h)$ and a self map $f: G \to G$ such that 
\begin{enumerate}
	\item the restriction of $f$ to each edge to either an immersion or a constant map.
	\item if $k: G \to R_0$ is a homotopy inverse of $h$ then $(h \circ f \circ k)_* \in \Phi$.
\end{enumerate}
\end{defn}

\begin{defn}[irreducible maps]
A \emph{core} graph $H$ is a graph, all of whose vertices have valence $\geq 2$. Let $g:G \to G$ be map that restricts on each edge to either an immersion or a constant map. $g$ is \emph{reducible} if there is a proper, nonempty core subgraph $H$ of $G$ which is invariant under $g$. If $g$ is not reducible it is called \emph{irreducible}.
\end{defn}

\begin{theorem}\cite{BH}
If $\Phi \in \out$ is irreducible then $\Phi$ has an irreducible train-track representative $(G,h,f)$.
\end{theorem}
 
\begin{defn}
Given a train-track representative $f: G \to G$ one can endow $G$ with a metric $\ell:E(G) \to (0,1)$ so that $f$ stretches each edge of $G$ by the same amount $\lam>1$. $\lam$ is called the \emph{expansion factor} of $f$, or the Perron-Frobenius eigenvalue of $f$. Even though a train-track representative of an irreducible outer automorphism is not unique, every such representative has the same stretch factor which we will associate to $\Phi$.
\end{defn}

\subsection*{Laminations of fully irreducible automorphisms}
Let $f:G \to G$ be a train-track representative of an irreducible outer automorphism $\phi$, let $c$ be the expansion factor of $f$. By replacing $f$ with a power if necessary, we may assume that $f$ has a fixed point $p$ in the interior of an edge. Let $I$ be an $\eps$ neighborhood of $p$ so that $f(I) \supset I$. Choose an isometry $\lam:(-\eps, \eps) \to I$ and extend uniquely to a local isometric immersion $\lam: \RR \to G$ so that $\lam(c^mt) = f^m(t)$ for all $t \in \RR$. $\lam$ is \emph{a periodic leaf in the lamination} $\Lam^+_f(G)$ (for a definition of $\Lam^+_f$ see \cite{BFH}). {\it A stable leaf subsegment} is the restriction of $\lam$ to a subinterval of $\RR$. Given a different metric graph $H \in \os$ and a homotopy equivalence $g: G \to H$, $\Lam^+_f(H)$ the attracting lamination in the $H$ coordinates is the collection of immersions $[g \lam]$ pulled tight. This definition does not depend on the train-track representative, so we can denote it by $\Lam^+_\phi$. An important feature of the leaves of $\Lam^+_\phi$

\begin{prop}[\cite{BFH} Proposition 1.8]\label{quasi_periodic}
Every periodic leaf of $\Lam^+_\phi$ is quasi-periodic.
\end{prop}

This means that for every length $L$ there is a length $L'$ such that
if $\al, \beta \subseteq \lam$ are subleaf segments with $length(\al) =
L$ and $length(\beta)>L'$ then $\beta$ contains an occurrence of
$\al$. One can think of $\lam$ as a necklace made of beads. The
segments of length $L$ that appear in $\lam$ are beads of different
colors. The proposition tells us that in any subchain of $\frac{L'}{L}$ consecutive beads we can find beads of all possible colors.

\subsection*{Geometric and Nongeometric Automorphisms}\label{geom}

\begin{defn}
An outer automorphism $\Phi$ of $F_n$ is \emph{geometric} if there is a surface automorphism $F:S_{g,p} \to S_{g,p}$ with $\pi_1(S_{g,p}) = F_n$ so that $F_* \in \Phi$.
\end{defn}

When $n=2$ all elements of $\out$ are geometric. This is false for $n>2$. In this section we describe when an irreducible automorphism is geometric.

\begin{notation}
Consider a graph $G$ if $\al$ is a loop we denote by $[\al]$ the immersed loop which is freely homotopic to $\al$. If $\al$ is a path then $[\al]$ is the immersed path homotopic to $\al$ relative to its endpoints.
\end{notation}

\begin{defn}[Nielsen Paths]
\label{Nielsen}Let $f:G \to G$ be a map, a {\it Nielsen path} of $f$
is a path $\beta$ such that $[f(\beta)] = \beta$. $\gamma$ is a \emph{pre-Nielsen path} if $[f^i(\gamma)]$ is a Nielsen path for some $i$. Note that if $\al,\beta$ are Nielsen paths then so is $\al\beta$. A Nielsen path is \emph{indivisible} if it cannot be expressed as a concatenation of other Nielsen paths. 
\end{defn} 

\begin{defn}
Let $f:G \to G$ be a train-track map. 
$f$ is \emph{stable} if it has no more than one indivisible Nielsen path.  
\end{defn}

\begin{theorem}\cite{BH}
Every irreducible $\Phi \in \out$ has a stable irreducible train track representative.
\end{theorem}

\begin{theorem}\cite{BH}
Let $\Phi$ be an irreducible outer automorphism and $f:G \to G$ a stable train-track representative for $\Phi$. $\Phi$ is geometric if and only if $f$ has a Nielsen loop $\beta$ which crosses every edge exactly twice. 
\end{theorem}

If $F: S_{g,p} \to S_{g,p}$ represents $\Phi$ then one of the boundary components will be a Nielsen loop of the type described in the theorem. If $f: G \to G$ is a train-track map with a Nielsen loop $\beta$ as described in the theorem then one can attach an annulus along one of its boundary components to $G$ along $\beta$ and get a surface $S_{g,1}$ and an induced map $f$ on it which represents $\Phi$. $\beta$ will correspond to the boundary on this surface.



\section{The Lipschitz metric on $\os$}

Let $(G,f,\ell),(G',f',\ell')$ represent two points $x,y$ in
$\X_n$. A {\it difference of markings} is a map
$h:G \to G'$ with $h \circ f$ homotopic to $f'$. We will assume
that $h$ is Lipschitz. By $\Lip(h)$ denote the Lipschitz constant of $\phi$, i.e. the smallest number so that $d_{y}(\phi(p),\phi(p')) \leq \Lip(\phi) \cdot d_x(p,p')$ for all $p,p' \in G$. Define the distance
$$d(x,y)=\min_h \log\Lip(h)$$
where $\min$ is taken over all differences of markings (it is attained
by Arzela-Ascoli). 

We claim that we may restrict our attention to maps $h$ that are linear on edges, since the minimum is realized by a linear map: For any map $h$ one
can construct a map $k \sim h$ which is linear on edges. Define $k(v)=h(v)$ on every vertex $v$ of $G$ and let the image of the edge
$(v,w)$ under $k$ be the immersed path $[h(v),h(w)]$, which is homotopic to $\im(h|_{[v,w]})$ rel endpoints and parameterized at a constant speed. It is clear that $\Lip(k) \leq \Lip(h)$. Therefore, we can usually restrict our attention to linear maps.

\begin{lemma} The Lipschitz distance satisfies the following:
\begin{enumerate}
\item $d(x,y)\geq 0$ with equality only if $x=y$.
\item $d(x,z)\leq d(x,y)+d(y,z)$ for all $x,y,z\in\X_n$.
\item $d$ is a geodesic metric; for any $x,y$ there is a path from $x$
  to $y$ whose length is $d(x,y)$. 
\end{enumerate}
\end{lemma}

\begin{proof}
\begin{enumerate}
\item Let $h$ be the linear map realizing $d(x,y)$. Since $h$ is a difference in marking, it is a homotopy equivalence. If $h$ were not onto, then $h_*$ would not be an isomorphism (since there are no edges in $y$ that can be homotoped away from every loop, recall that there are no free edges). Thus, $vol(y) \leq vol( \text{Im}(h)) \leq \Lip(h) vol(x)$ (because $h$ stretches all edges by at most $\Lip(h)$). Since $vol(y) = vol(x) = 1$ we get $\Lip(h) \geq 1$ hence $d(x,y) \geq 0$. We get equality if only if $\Lip(h) =1$ and $h$ is a local isometry, hence if and only if $h$ is an isometry. This implies $x=y$.
 
\item Suppose $h:x \to y$, $k:y \to z$ realize the distance, we shall also call them {\it optimal maps}, then $j =
    k \circ h$ is a difference in marking from
    $x$ to $z$, thus by the chain rule $\Lip(h) \Lip(k) \geq \Lip(j) 
    \geq \exp(d(x,z))$. Taking $\log$ we get $d(x,y)+d(y,z) \geq d(x,z)$.
\item See e.g. \cite{FM}. 
\end{enumerate}
\end{proof}

\subsection{Candidates and Computing Distances}\label{candidates}

\begin{defn}
Suppose $(G,f,\ell)$ represents $x$ in $\os$. Let $\al$ be a loop in $G$. $\ell(\al)$ is the length of an immersed loop $[\al]$ freely homotopic to $\al$ in $x$.
\end{defn}

We say that a loop $\alpha$ in $G$ is a {\it candidate} if either
\begin{itemize}
\item it is embedded, or
\item (figure eight) there are two embedded circles $u,v$ in $G$
  that intersect in one point and $\alpha$ crosses $u,v$ once and does
  not cross any edges outside of $u$ and $v$, or
\item (barbell) there are two disjoint embedded circles $u,v$ in
  $G$ and an arc $w$ that connects them, and whose interior is
  disjoint from $u$ and $v$, and $\alpha$ crosses $u,v$ once, $w$
  twice, and no edges outside $u\cup v\cup w$.
\end{itemize}

The following is due to Tad White, and first written in \cite{FM}. We give a shorter proof here.

\begin{prop}\label{green graph}
Let $x,y\in \X_n$, $x=(G,f,\ell), y=(G',f',\ell')$ and let
$g:G \to G'$ be a difference of markings. 
$$d(x,y)=\log \inf_{\al} \frac{\ell'(g(\alpha))}{\ell(\alpha)}$$
Where the infimum is taken over all loops $\al$ in $x$. Moreover, there is a candidate loop $\alpha$ in $G$ which realizes the infimum.  
\end{prop}

\begin{defn}[a t-t structure induced by a map]\label{induced t-t structure}
Let $g:G \to H$ be a map whose restriction to each edge is either an immersion or constant. \emph{The train-track structure induced by} $g$ is the following equivalence relation: $e_1 \sim e_2$ if they have the same initial vertex and if there are small enough  initial segments of $g(e_1)$ and $g(e_2)$, which coincide (they define the same germ).
\end{defn}

\begin{remark}
Notice that the train-track structure defined in \ref{induced t-t structure} is smaller than the one in definition \ref{self-induced tt structures}. When we refer to a t-t structure defined by $g$ and $g$ is a self-map we will always mean the structure defined in \ref{self-induced tt structures}.
\end{remark}

\begin{proof}[Proof of  Proposition \ref{green graph}]
Let $\al$ be an immersed loop in $G$. The loop $g(\al)$ might
not be immersed. 
\begin{equation}\label{eq1}
l'(g(\al)) \leq \Lip(g) l(\al)
\end{equation}
This is a strict inequality if one of the edges that $\al$ crosses is stretched less than $\Lip(g)$ or if the loop $g(\al)$ is not immersed. We have 
\begin{equation}\label{eq1a}
d(x,y) \geq \log \inf_{\al} \frac{\ell'(g(\alpha))}{\ell(\alpha)}
\end{equation}

We define the {\it green graph} of $G$ with respect to $g$ as the set of edges on which the slope of $g$ is $\Lip(g)$. We denote it by $G_g$. 
There is a train-track structure on $G_g$ induced by $g$ as in definition \ref{induced t-t structure}. 

Now suppose $g$ is an optimal map, i.e., satisfies $d(x,y) = \log \Lip(g)$, and suppose that $G_g$ is the smallest among all optimal maps. To show equality in \ref{eq1a} we will find a candidate loop $\al$ that is contained in $G_g$ and is legal with respect to the train-track structure of $g$. \\

First, we claim there are at least two 
gates at each vertex. See Figure \ref{optimal_map} 
for an example. Suppose by way of contradiction that there is a vertex $v$
which has only one gate, we will construct a new map $\psi$ homotopic to $g$ so that either $\Lip(\psi)<\Lip(g)$ or $\Lip(\psi) = \Lip(g)$ and $G_\psi \subsetneq G_g$ which gives a contradiction. Define a map $\psi$ by $\psi(u)=g(u)$ for all vertices $u \neq v$. To define $\psi(v)$, recall that all images of edges in $G_g$ adjacent to $v$ start with the same initial subsegment in $G'$ (there is one gate). Take the subsegment to be small enough as to not contain a vertex. Let $\psi(v)$ be the point on this subsegment, a distance $\varepsilon$ (to be chosen later) away from $g(v)$. Define $\psi$ to be homotopic to $g$ and linear on edges. 

 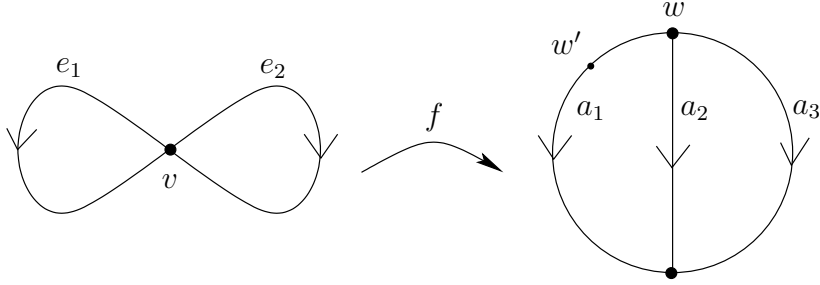
\begin{figure}[ht]
 \begin{center}
 \input{optimal_map.pstex_t}
 \caption{\label{optimal_map} The following is an example of a map $f$
  where $X_f$ has a vertex with one gate. We show that the map is not
  optimal. In the graph on the left both edges have length $\frac{1}{2}$
   and in the graph on the right all three edges have length
   $\frac{1}{3}$. Suppose the map $f$ takes $e_1 \to a_2
   \overline{a_3}$ and $e_2 \to a_1 \overline{a_2}
   a_1 \overline{a_3} a_2 \overline{a_1}$. The stretch of $f$ on $e_1$ is
   $\frac{2/3}{1/2} = \frac{4}{3}$ and the stretch of $f$ on $e_2$ is
   $\frac{6/3}{1/2} = 4$, so $\Lip(f) = 4$ and the green graph of $f$ is $X_f $ which is just $e_2$. Both $e_2, \overline{e_2}$ begin with $a_1$ so $X_f$ contains
   only one gate at $v$. Let $w'$ be a point on $a_1$ which is
   $\ve$ away from $v$ where $8\ve< 4-\frac{4}{3}$. $w'$ divides $a_1$
   into two edges $b_1, b_2$. Consider the map $f_1$ which takes $e_1
   \to \overline{b_1} a_2 \overline{a_3} b_1$ and  $e_2 \to b_2 \overline{a_2}
   b_1 b_2 \overline{a_3} a_2 \overline{b_2}$. $f_1$ is homotopic to
   $f$. $f_1$ stretches $e_1$ by $\frac{2/3+2\ve}{1/2} =
   \frac{4}{3}+4\ve$ and $e_2$ by  $\frac{2-2\ve}{1/2} = 4-4\ve$. Since $\ve$ is small enough 
   $\Lip(f_1) = 4-4\ve$ which is smaller than $\Lip(f)$.
 } 
 \end{center}
 \end{figure}

For $e \in G$ which is not adjacent to $v$ the images of $e$ under $g$ and $\psi$ coincide so the slope of $\psi$ on $e$ is still $\leq \Lip(g)$. If $e \in G_g$ is adjacent to $v$, we have made $g(e)$ shorter by $\ve$ so the slope of $\psi$ on $e$ is strictly smaller than $\Lip(g)$.  Let $S$ be the second largest slope of $g$ in $G$. If $e \in G \setminus G_g$ then $\ell'(\psi(e)) \leq S \cdot \ell(e) + 2\ve$. So if we take 
\[ \ve < \frac{1}{2}( \Lip(g) - S )(\text{length of shortest edge in } G)\] 
then the slope on $e$ will still be strictly smaller than $\Lip(g)$. This proves the claim.

Construct a legal loop in $\al \in G_g$ as follows. Start constructing an embedded legal path $\al$ until it intersects itself. That is, until $\al(t_1) = \al(t_2)$ for some $0 \leq t_1<t_2$ (we now consider $\al$ as a map from $[0,L]$ to $G$). If the turn $\{ D_+\al(t_1), D_-\al(t_2) \}$ is legal, then $\al|_{[t_1,t_2]}$ is a legal loop. If it is illegal, let $\al'(t) = \al(t + t_1)$ and rename $\al, t_1$ to get $\al(0)= \al(t_1)$ and $\{ D_+\al(0), D_-\al(t_1) \}$ is illegal. By the previous paragraph there is another gate at $\al(0)$ in $G_g$. Extend $\al$ to cross this gate and continue until there are $t_2<t_3$ so that $\al(t_2) = \al(t_3)$. If $0 < t_2 \leq t_1$ then either $\al|_{[t_2,t_3]}$ is a legal loop or $\al|_{[0,t_2]} \cup \al^{-1}|_{[t_3,t_1]}$ is a legal loop. If $t_1<t_2<t_3$ then either $\al|_{[t_2,t_3]}$ is a legal loop or $\al|_{[0,t_4]} \cup \al^{-1}|_{[t_1,t_2]}$ is a legal loop (barbell or figure 8 loops). 
\end{proof}

Note that $d(x,y)\geq \log\frac{\ell'(g(\alpha))}{\ell(\alpha)}$
for {\it any} loop $\alpha$ and any difference in marking $g$. The right hand side does not depend on a
particular choice of $g$, so one can effectively compute the
distance by maximizing the ratio over the finitely many candidate
loops. \\

For $x \in \os$ represented by $(G,f,\ell)$, any conjugacy class $\al$ of $F_n$ may be identified with an immersed loop $\al_G$ in $G$. We will use the same notation for both the conjugacy class and the loop representative. If we want to emphasize that the loop $\al$ is in the graph $G$ we will denote it by $\al_G$

\begin{defn}
We say that the conjugacy class $\al$ is a basis element if $\{ \al \}$ can be completed to a basis of $F_n$.
\end{defn}

\begin{prop}\label{can_are_basis}
Let $\al_X,\beta_X$ be different candidates in $(X,f,\ell)$. Then there is a third candidate $\gamma_X$ so that $\{ \al, \gamma \}$ and $\{ \beta, \gamma \}$ can each (separately) be completed to a basis of $F_n$.
\end{prop}
\begin{proof}
Suppose $\al_X$ is a candidate and $\gamma_X$ is an embedded loop such that $\gamma_X \setminus \al_X \supseteq \{ e_i \}$. 
Let $J$ by a maximal forest in $X$ which doesn't contain 
$e_i$. Collapse $J$ to get $R_X$ a
wedge of circles. Since $e_i$ was not collapsed, $
\gamma_R \setminus  \al_R \supseteq \{ e_i \}$. Let $e_j$ be any edge that $\al$ crosses exactly once then $< \al_R, \gamma_R, e_1, \dots, \hat{e_i}, \dots , \hat{e_j}, \dots e_n
>$ represents a basis for $F_n$ because $< e_j, e_i, e_1, \dots, \hat{e_i}, \dots , \hat{e_j}, \dots e_n>$ is a basis for $F_n$.\\ 
Suppose $\al$ is a figure 8 or barbell candidate and $\gamma$ is an embedded circle so that $\gamma \subseteq \al$. Choose an edge $e_i$ in $\al \setminus\gamma$ which $\al$ crosses only once. Collapse a maximal forest that doesn't contain $e_i$. Now choose $e_j$ in $\gamma_R$. Then $<\al_R, \gamma_R, e_1, \dots, \hat{e_i}, \dots , \hat{e_j}, \dots e_n>$ is a basis for $F_n$. \\
\indent Let $\al_X, \beta_X$ be any two candidates. If one of them is an embedded loop whose image is not contained in the other then $\al,\beta$ can be completed to a basis. If $\al, \beta$ have the same image then find an embedded loop $\gamma$ as in the previous paragraph so that $\al, \gamma$ and $\beta, \gamma$ can be completed to a basis. If they have different images and are not embedded, let $\gamma$ be an embedded loop so that $ \im \gamma \subseteq \im \al$. Then $\al,\gamma$ can be completed to a basis. If $\im \gamma \subseteq \im \beta$ then $\beta, \gamma$ can be completed to a basis. If $\gamma$ is not contained in $\beta$ then again by the previous paragraph $\beta, \gamma$ can be completed to a basis. 
\end{proof}

\subsection{Asymmetry of $d(x,y)$}

In general, the lipschitz distance is not symmetric as is shown in figure \ref{non_symmetric}. 

 \begin{figure}[ht]
 \begin{center}
 \input{non_symmetric.pstex_t}
 \caption{\label{non_symmetric}An example where $d(X,Y)=\log \frac{2m-2}{m} \sim \log 2$ and
   $d(Y,X)= \log \frac{m}{2}$}
 \end{center}
 \end{figure}
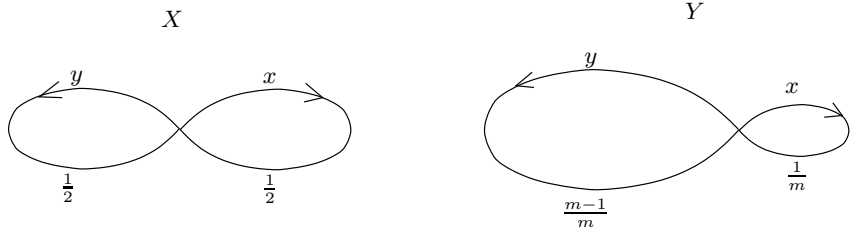

We will make use of the following result: 

\begin{definition}
For $\theta>0$ the $\theta$-thick part of $\os$ is
\[ \os(\theta) = \{ (X,f,l) \in \os \mid l(\al) \geq \theta \text{
  for all } \al \} \]
\end{definition}

\begin{theorem}\label{almost_sym}\cite{AKB}
For any $\theta>0$ there is constant $c = c(\theta,n)$ such that:
\[ d(x,y) \leq c \cdot d(y,x)  \] for any
$x,y \in \os(\theta)$.
\end{theorem}


\section{The axis of an irreducible automorphism and its projection}

It is straightforward to check that the right action of $\out$
on $\os$ is an isometric action.

\begin{notation}
Let $x = (G,g,\ell)$ be a point in $\os$. In this section and in later chapters, we will repress $g$ and $\ell$ and denote $x$ with $G$. 
\end{notation}

Let $\phi$ be an irreducible outer automorphism. Observe that if $f: G  \to G$ is a train-track representative for $\phi$ then so is $f: G \cdot \phi  \to G \cdot \phi$. We get a sequence of points $G \cdot \phi^m$ so that $f:G \cdot \phi^m \to G \cdot \phi^m$ is a train-track representative for $\phi$. We would like to define a line $\LL_f  = \{ G_t \}_{t \in \RR} = \{ G(t) \}_{t \in \RR}$ so that: 
\begin{enumerate}
	\item $G_t$ is a directed geodesic parameterized according to arc-length in the Lipschitz metric. That is, for $t < t' $ we have $d(G_t, G_{t'}) = t' - t$. In particular, for $t< t' < t''$: $d(G_t, G_{t''}) = d(G_t, G_{t'}) + d(G_{t'}, G_{t''})$
	\item $G_0 = G$
	\item $\LL_f$ is invariant under $\phi$. 
	\item\label{ttmaps} For each $t$, there are maps $f_t: G_t \to G_t$ that are irreducible train-track representatives of $\phi$.
\end{enumerate}

The way to achieve this is to start ``folding" $G$ onto itself by identifying appropriate segments in edges which form an illegal turn until we reach $G \cdot \phi$. This defines a path $[G , G \cdot \phi] := \{ G_t \}_{ 0 \leq t \leq \log \lam}$ where $\log \lam = d(G, G \cdot \phi)$. Then we translate this path using $\phi$ to construct a line $\displaystyle \LL_f = \bigcup_{m = - \infty} ^{\infty} [G, G \cdot \phi] \cdot \phi^m$. This line is automatically invariant under $\phi$. It is not hard to see that it is a directed geodesic. For more details see \cite{AKThesis}, \cite{HM} or \cite{FM}.

\subsection{The projection to an Axis}

\begin{notation}
In this section we denote the point $x=(G,h,\ell) \in \os$ by $G$. The length of the (immersed) loop $\al$ in $G$ will now be denoted by $l(\al, G)$.
\end{notation}

Let $\phi$ be an outer automorphism, and suppose $f:G \to G$ is a 
train-track representative for $\phi$ and $g:H \to H$ is a train-track
representative for $\phi^{-1}$. We may suppose $f$ and $g$ are stable. Let $\LL_f  = \{ G(t) \}_{t \in \RR}$ be an axis for $\phi$, and $\LL_g = \{H(t)\}_{t \in \RR}$ an axis for $\phi^{-1}$. Let $\lam, \nu$ be the expansion factors of $\phi, \phi^{-1}$. We will show that for a conjugacy class of basis elements $\al$ there is a $k_0$ such that $l(\phi^{k-k_0}(\al),G) \sim \lam^{k-k_0} l_0 + \mu^{k_0-k} l_0$. This $k_0$ will allow us to define the projection to $\LL_f$.\\

\begin{notation}
Fix a graph $G$ if $\al$ is a loop we denote by $[\al]$ the immersed loop which is freely homotopic to $\al$. If $\al$ is a path then $[\al]$ is the immersed path homotopic to $\al$ relative to its endpoints.
\end{notation}

\begin{defn}[The bounded cancellation constant]
Let $f: T \to T'$ be an equivariant map of free metric $\RR$-trees. There is a constant denoted $BCC(f)$ so that for all $p,q,r \in T$ such that $p \in [q,r]$ we have $d(f(p), [f(q), f(r)]) < BCC(f)$. This was first observed  by Cooper \cite{C}. $BCC(f)$ is called \emph{the bounded cancellation constant} of $f$.  
\end{defn}

We need the following lemmas and notions due to Bestvina, Feighn,
and Handel \cite{BFH}. Consider a path $\delta = \al \cdot \beta \cdot \gamma$ where $\al,\beta,\gamma$ are legal paths with respect to $f$ but the turn denoted by the dot might be illegal. $f(\delta) =  f(\al)\cdot f(\beta) \cdot f(\gamma)$ and the cancellation only occurs at the dot.  $[f(\delta)]$ will contain $\theta$, the subpath of  $f(\beta)$ obtained from $f(\beta)$ by truncating paths of length $BCC(f)$ at both ends. If $len(\beta)> \frac{4 BCC(f)}{\lambda-1}$ then $len(\theta)> len(f(\beta)) - 2 BCC(f) > \frac{\lam+1}{2} len(\beta)$. 

\begin{definition}\label{def_legality}
We call the constant $\kappa = \frac{4 BCC(f)}{\lam-1}$ the legality threshold (note the slight difference from the ``critical constant" of \cite{BFH}).  The
legality of $\delta$ with respect to the train-track structure of $f$ is
\[ LEG_f(\delta,G) =\frac{\text{Total length of all legal pieces of length
} > \kappa}{l(\al,G)} \]
If $LEG_f(\delta,G)>\eps$ we say that $\delta$ is $\eps$-legal.
\end{definition}

\begin{lemma}\label{lemma_exp_growth_1}
For every $\eps>0 $ there is a constant $C=C(\eps)$ so that if $\delta$ is $\eps$-legal then $l(f^n(\delta),G) > C \lam^{n} l(\delta,G)$. 
\end{lemma}

\begin{proof}
Let $\beta_1, \dots, \beta_k$ be legal subsegments of $\delta$ of length $> \kappa$. Then $[f(\delta)]$ contains $\theta_1, \dots, \theta_n$ the middle subsegments of $f(\beta_1), \dots , f(\beta_n)$ after truncating $BCC(f)$ from either side. $l(\theta_i,G) > \frac{\lam+1}{2}l(\beta_i,G)$ Therefore, \[ \begin{array}{lllll}
l(f^n(\delta),G) & > & \sum_{i=1}^k l(\theta_i, G) & > & \frac{\lam+1}{2} \sum_{i=1}^k l(f^{n-1}(\beta_i),G)) \\[0.2 cm]
&&& =& \frac{\lam+1}{2} \lam^{n-1} \sum_{i=1}^k l(\beta_i,G) \\[0.2 cm]
&&& > &\frac{\lam+1}{2\lam} \lam^n \eps \phantom{a} l(\delta,G)
\end{array} \]
Therefore $C(\eps) = \frac{\lam+1}{2\lam} \eps$.
\end{proof}

\begin{lemma}[Lemma 5.6 in \cite {BFH}]
If $\phi$ is non-geometric there is a constant $\eps_0>0$ and an
integer $N$ such that for any conjugacy $\al$:
\begin{center} 
$LEG_f(\phi^N(\al),G)> \eps_0$ or $LEG_g(\phi^{-N}(\al),H)>
\eps_0$
\end{center}
\label{BFH_lemma1}\end{lemma}

We want some version of this lemma for geometric
automorphisms. In the proof of
\ref{BFH_lemma1} the assumption that 
$\phi$ is nongeometric is used only to bound the number of
consecutive Nielsen paths appearing in $\al_G$. \\

Recall the definitions of  Nielsen and pre-Nielsen paths from Definition \ref{Nielsen}. Recall that if $\phi$ is 
geometric and fully irreducible and $f$ is stable, then
the unique indivisible Nielsen path $\beta$ is a loop which crosses every edge exactly twice. Therefore, in this case there is no bound for the number of consecutive Nielsen paths that appear in $\al_G$. Indeed the statement in Lemma \ref{BFH_lemma1} is false for $\beta$ if $\kappa$ is big enough. We shall restrict our attention henceforth to the case where $\al$ is the conjugacy class of a basis element. 

\begin{prop}\label{NP_in_can}
There is a bound $K$ that depends only on $\phi$, such that if $\al$ is a conjugacy class of a basis element, then $\al_G$ cannot cross more than $K$ consecutive pre-Nielsen paths.
\end{prop}

\begin{proof}
The case where $\phi$ is nongeometric is handled in \cite{BFH}.\\
If $\phi$ is geometric then $f:G \to G$ may be extended to an automorphism of a surface of genus $g$ and one boundary component $F:S_{g,1} \to S_{g,1}$ where $G$ is a spine for $S_{g,1}$ (See \cite{BH}). In this case, the indivisible Nielsen path $\beta$ can be represented by the boundary circle. Collapse $G$ to a rose $R$ with vertex $*$ (also embedded in the surface) and notice the Whitehead graph of $\beta$ with respect to the basis represented by the edges of $R$ is equal to $\text{Link}(*, S_g)$ which is a circle. Thus, $Wh_R(\beta)$ is connected with no cut vertex. Now if $\al_G$ crosses two pre-Nielsen loops consecutively then 
for some $m$: $f^m(\al)$ crosses $\beta$ twice consecutively. Therefore, $Wh_R(f^m(\al)) \supseteq Wh_R(\beta)$ so $Wh_R(f^m(\al))$ is connected with no cut-vertices. By Whitehead's theorem $f^m(\al)$ is not a basis element, which is a contradiction. 
\end{proof}

\begin{lemma}\label{pos_legality}
For any irreducible outer automorphism $\phi$ there is a constant
$\eps_0>0$ and an integer $N$ such that for any basis element $\al$,  
\[ \begin{array}{lll}
LEG_f(\phi^N(\al),G)> \eps_0 & \mbox{ or } &
LEG_g(\phi^{-N}(\al),H)> \eps_0 
\end{array} \]
if $\phi$ is nongeometric this holds
for all $\al$. In particular, the above holds for a conjugacy class $\al$ that is represented by a candidate loop in some $x \in \os$.\label{BFH_lemma1b}\end{lemma}
\begin{proof}
There is an integer $K$ so that for any basis element $\al$, $\al$ does not contain $K$ consecutive Nielsen paths. Apply the proof of Lemma \ref{BFH_lemma1} as it appears in \cite{BFH} to get the first part of the statement. By Proposition \ref{can_are_basis} all candidates are basis elements so the statement applies to them. 
\end{proof}

Fix a conjugacy class of a basis element $\al$. Notice that if $LEG_f(\phi^N(\al),G) \geq \eps$ then
$LEG_f(\phi^m(\al),G) \geq \eps$ for all $m>N$. Define
\[ \begin{array}{l}
k_0  = k_0(\al) = \max \{k | LEG_f(\phi^k(\al),G)<\eps_0 \}\\[0.1 cm]
 k'_0 = k'_0(\al) = \min \{k | LEG_g(\phi^k(\al),H)<\eps_0 \}
 \end{array}
 \] 
To see the existence of $k_0$ recall that for all basis elements $\al$, the weak limit $\lim_{n \to \infty}\phi^n(\al)$ is the attracting lamination and the weak limit $\lim_{n \to -\infty}\phi^n(\al)$ is the repelling lamination (e.g., \cite{BFH}). Therefore there is some $k$ such that $\phi^k(\al)$ is $\eps_0$-legal and $\phi^{-k}(\al)$ is not $\eps_0$-legal in $G$. Since the roles of $f,g$ are symmetric, applying the same argument to $g$ shows the existence of $k_0'$. 

\begin{lemma}\label{k_k'}
There is an $N$ so that for all $\al$, $|k_0(\al) - k'_0(\al)|<N$
\end{lemma}
\begin{proof}
At $k_0$,  $LEG_f(\phi^{k_0}(\al),G)
\ngtr \eps_0$, thus by Lemma \ref{pos_legality}, $LEG_g(\phi^{k_0-2N}(\al),H)> \eps_0$. This implies that $k_0' > k_0 -2N$. By symmetry we get the result.
\end{proof}

For each conjugacy class $\al$ define 
\[ l_\al(t) = l(\al,G(t))\] 
We now show that if $\al$ is a basis element, then there is a bounded set on which $l_{\al}(t)$ achieves its minimum, 
the bound is uniform over all conjugacy classes $\al$. \\

\begin{definition}
Let $t_0(\al) = k_0 \log \lam$ and $t'_0(\al) = k'_0 \log \lam$.   
\end{definition}

\begin{lemma}\label{exp_growth2} There exists a $C$ such that for every basis element $\al$ let $n(t) = \lfloor \frac{|t|}{\log \lam} \rfloor$ then for $t>0$ 
\begin{equation}\label{eq11}
\frac{1}{C} \cdot \lam^{ n(t)} \cdot l_{\al}(t_0) < l_\al(t_0 + t) \leq C \cdot \lam^{ n(t)} \cdot l_{\al}(t_0) 
\end{equation}
and for $t<0$
\begin{equation}\label{eq2}
\frac{1}{C} \cdot \mu^{ n(t)} \cdot l_{\al}(t_0) < l_\al(t_0 + t) \leq C \cdot \mu^{ n(t)} \cdot l_{\al}(t_0) 
\end{equation}
\end{lemma}

\begin{proof}
The right-hand inequality of \ref{eq11} is obvious. To get the left-hand side, first notice:
\[ \frac{l(\al,G(t_0 + t)}{l(\al, G(t_0 + n(t) \log \lam))} \leq e^{d(G(t_0 + n(t) \log \lam), G(t_0 + t))}  = e^{t - n(t) \log \lam} \leq \lam\]
Apply Lemma \ref{lemma_exp_growth_1} for $\eps = \eps_0$ and notice that $\al$ is $\eps_0$-legal at $t_0+\log \lam$ thus
\[ \begin{array}{lll}
 l_\al(t_0+t) & = & l(\al, G(t_0+t)) > \lam^{-1} l(\al,G(t_0+ n(t)\log\lam)) \\[0.1 cm]
 &  =  &\lam^{-1} l(\phi^{n(t)}(\al),G(t_0))  = \lam^{-1} C_{\ref{lemma_exp_growth_1}} \lam^{n(t)} l(\al,G(t_0)) \\[0.1 cm]
 & = & \frac{C_{\ref{lemma_exp_growth_1}}}{\lam} \lam^{n(t)} l_\al(t_0)
 \end{array} \] 
To prove equation \ref{eq2}, first note  
\[ l_\al(t_0 + t) = l(\al, G(t_0+t)) \leq \lam \cdot l(\al, G(t_0- (n(t)+1) \log \lam ) = \lam \cdot l(\phi^{-(n(t)+1)}(\al), G(t_0)) \]
Let $D = \max\{ d(G(t_0), H(t'_0)), d(H(t'_0), G(t_0)) \}$  this is bounded independently of $\al$ since $|k_0 - k'_0|<N$ and because $d(G(t),H(t))$ is bounded independently of $t$ (by periodicity). Thus,
\[l(\phi^{-(n(t)+1)}(\al), H(t_0')) \geq \frac{1}{e^D} l(\phi^{-(n(t)+1)}(\al), G(t_0)) \geq  \frac{1}{\lam e^D} l_\al(t_0+t)\]
Now 
\[ l(\phi^{-(n(t)+1)}(\al), H(t_0')) \leq \mu^{n(t)+1}l(\al,H(t_0') \leq  \mu^{n(t)+1}e^D l(\al, G(t_0)) \]
We get the right inequality in \ref{eq2}. The left inequality is proven similarly. 
%
\end{proof}

\begin{definition}[min set]
For a conjugacy class $\al$ of a basis element: let $L = \min \{ l_\al(t) \mid t \in \RR
\}$ and denote by $T_\al$ the set of $t_\al$ such that $l_\al(t_\al)
= L$. The min set of $\al $ is $\pi_f(\al) = \{ G(t_\al) \mid t_\al \in
T_\al \}$.
\end{definition}

It follows from Lemma \ref{exp_growth2} that 

\begin{cor}\label{min_close_to_legal}
There exists an $s>0$ so that for any basis element $\al$ and for all $t_\al \in T_\al$, $|t_\al-t_0|<s$. 
\end{cor}

Which implies 

\begin{cor}\label{coarse_minimum}
There is an $s>0$ such that for every basis element $\al$, 
$\diam\{ T_\al \}<s$ hence $\diam\{ \pi_f(\al)\}$ is bounded independently of $\al$.
\end{cor}

From now on $t_\al$ denotes any choice of element in $T_\al$, for example the smallest one. 
The following proposition follows from corollary \ref{min_close_to_legal} and Lemma \ref{k_k'}.

\begin{cor}\label{close_minima}
There is an $s>0$ such that for every primitive $\al$: 
\[ \forall G' \in \pi_f(\al) \phantom{bl} \forall H' \in
\pi_g(\al): \phantom{blah} d(G',H')<s \]
I.e. the min sets of $\al$
with respect to $\LL_f$ and $\LL_g$ are uniformly close.
\end{cor}


\begin{cor}\label{pos_leg}
There is an $s>0$ such that for every basis element $\al$, if $t>t_\al+s$ then $LEG(\al,G(t))>\eps_0$ (the legality is computed with respect to the train track structure induced by $f_t: G(t) \to G(t)$ (see Definition \ref{self-induced tt structures}) from item \ref{ttmaps} in the list of properties of $\LL_f$).  
\end{cor}

The following observation states that if $\al$ is almost legal in $G(t)$, then it almost realizes the distance $d(G(t), G(t+t'))$. Denote $\ST_\al(X,Y) = \frac{l(\al, Y)}{l(\al, X)}$.

\begin{prop}\label{alm_legal_almost_max_st}
There is a $C$ so that if $\al$ is $\eps_0$-legal in $G(t)$ with respect to $g_t$ then for all $t'>0$, $$\log \ST_\al(G(t),G(t+t')) - C \geq d(G(t), G(t+t')) = t'$$
\end{prop}
\begin{proof}
Since $\al$ is $\eps$-legal $t>t_0$. Let $K=C_{\ref{exp_growth2}}$ from  Lemma \ref{exp_growth2}. 
\[ \begin{array}{l}
	l_\al(t+t') \geq \frac{1}{K} \lam^{n(t+t')} l_\al(t_0)  \\[0.2 cm]
	l_\al(t) \leq K \lam^{n(t)} l_\al(t_0)
\end{array}\]
Then 
\[\ST_\al(G(t),G(t+t')) = \frac{l_\al(t+t')}{l_\al(t) }\geq \frac{1}{K^2} \lam^{n(t+t')-n(t)} \]
\end{proof}

Now we are in a position to define a coarse projection $\pi_f: \os
\to \LL_f$. Let $X \in \os$ and $T_X = \{ t \mid d(X,G(t))= d(X,\LL_f)
\}$. Define the projection of $X$ to $\LL_f$ by $\pi_f(X) = \{ G(t) \mid t \in T_X \}$.

\begin{prop}\label{coarse_proj}
There is an $s>0$ such that for every point $X \in \os$: $$\diam(\pi(X))<s$$
\end{prop}
\begin{proof}
Let $u(t)$ be a coarsely exponential function, i.e. a function that satisfies Lemma \ref{exp_growth2}. Let $C = C_{\ref{exp_growth2}}$, $s(u) = \frac{2\log(C)}{\log \lam}+1$. Let $t_0<t$ and $t'>t+s(u)$ then $\lam^{n(t') - n(t)}> C^2$. Thus by Lemma \ref{exp_growth2} applied to $u$ we have: 
\[ u(t') - u(t) > \frac{1}{C} \lam^{n(t')}u(t_0) - C \lam^{n(t)}u(t_0)>0 \]  
Similarly for $t'<t-s(u)$ and $t<t_0$ for some appropriate $s(u)$. To sum up:
\[\begin{array}{ll}
\text{If } t_0<t<t+s(u)<t' \text{ then } & u(t)<u(t') \\[0.2 cm]
\text{If } t_0>t>t-s(u)>t' \text{ then } & u(t)<u(t')
\end{array}
\]
Let $\al,\beta$ be two candidates in $X$. The function $u(t) = \ST(\al_t) = \frac{l_\al(t)}{l(\al,X)}$ differs from $l_\al(t)$ by the multiplicative constant $\frac{1}{l(\al,X)}$. So $u(t)$ and $v(t) = \ST(\beta_t)$ are also coarsely exponential. \\

We claim that $h(t) = \max \{ u(t), v(t) \}$ also has a coarse minimum. Let $t_u = t_0(u) = t_0(\al)$ and $t_v = t_0(\beta)$ and for concreteness, assume $t_u<t_v$. Note that $\lim_{t \to \pm \infty} h(t) = \infty$ so $h$ obtains a minimum at some $t_h$. \\ 
If $h(t_u) = u(t_u)$ let $t>t_u+s(u)$ then we have 
\[ \begin{array}{l}
h(t_u-t) \geq u(t_u-t)> u(t_u) = h(t_u)\\[0.2 cm] 
h(t_u+t) \geq u(t_u+t)> u(t_u) = h(t_u)
\end{array}\] 
hence the diameter of the min-set of $h$ is at most $2s(u)$. Similarly, if $h(t_v) = v(t_v)$ we are done. So we can assume $h(t_u) = v(t_u)>u(t_u)$ and $h(t_v) = u(t_v)>v(t_v)$. By continuity, there is a point $t_1$ so that $u(t_1) = v(t_1)=h(t_1)$. $h(t_1+t) \geq u(t_1 + t) > u(t_1)$ for $t>s(u)$ (since $u$ is increasing in this domain) and $h(t_1-t) \geq v(t_1-t) > v(t_1)$ for $t>s(v)$. Therefore the min-set of $h$ is bounded by $2(s(u)+s(v))$ see Figure \ref{coarse_proj_fig}. \\
Since there is only a finite number of candidates (depending only on $n$) then the diameter of $\pi(X)$ is uniformly bounded. 
\end{proof}

\begin{figure}[ht]
\begin{center}
\includegraphics{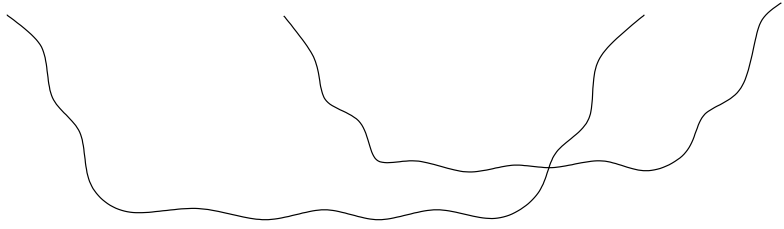}
\caption{\label{coarse_proj_fig}If two functions have a coarse minimum then their max has a
  coarse minimum.}
\end{center}
\end{figure}

Using the fact that the length map $l_\al(t)$ is coarsely exponential one could show that $\pi_f: \os \to \LL_f$ is ``coarsely Lipschitz". However we will get a better result in Corollary \ref{proj_Lip}.


\section{The Whitehead graph of the attracting and repelling laminations}

\subsection{The Length of a Lamination}

Let $\phi$ be a fully irreducible automorphism, and $f:G_0 \to G_0$ a
stable train-track representative of $\phi$. Let $T_0$ be the universal cover of
$G_0$, and $\wt{f}:T_0 \to T_0$ a lift of $f$. Suppose $\Lam^+_\phi(G_0)$
is the attracting lamination of $\phi$ realized as bi-infinite lines in
$G_0$. 

\begin{defn}[length of $\Lam^+_\phi$ in $T$ scaled with respect to $T_0$]\label{length of lam def} 
Given a metric tree $T$ in unnormalized Outer
Space (i.e., where we do not identify homothetic trees) one can define the length of the lamination $\Lam^+$ in $T$,
scaled with respect to $T_0$, as follows. Let $h:T_0 \to T$ be an equivariant Lipschitz map and let $\sigma$ be a
subsegment of the leaf $\lambda \in \Lam^+_\phi(G_0)$ and $\wt{\sigma}$ be
a lift of $\sigma$ to $T_0$. Let $[h(\wt{\sigma})]$ the tightened image of
$\wt{\sigma}$ in $T$ then
\begin{equation}
l_{T_0}(\Lam^+,T) = \displaystyle \lim_{\sigma \to \lambda}
\frac{l([h(\wt{\sigma})],T)}{l(\wt{\sigma},T_0)} \label{lam_len_eq}
\end{equation} 
\end{defn}

Note that Definition \ref{length of lam def} does not depend on the choice of leaf $\lam$ of $\Lam^+_\phi$ since all leaves of $\Lam_\phi^+$ have the same leaf segments. In Lemma \ref{lam_len_well_defined} we will prove that this definition is mathematically sound. But first, let us observe that the limit in equation (\ref{lam_len_eq}) is not invariant when resealing the metric of $T$. Therefore, we modify the definition for $[T]$
\begin{defn}[length of $\Lam^+_\phi$ in the projective class of $T$ scaled with respect to $T_0$]\label{len of lam in proj}
Let $[w_1],
\dots , [w_J]$ be a set of conjugacy classes in $F_n$ that cannot be
simultaneously elliptic in any very small $F_n$-tree (for example the set of words of length at most $2$ in some basis). Let $\tr(w_1,T), \dots ,\tr(w_J,T)$ be their
translation lengths in $T$ and $d(T) = \sum_{i=1}^{J}tr(w_i,T)$ then
\begin{equation}
pl_{T_0}(\Lam^+, [T]) = \frac{l_{T_0}(\Lam^+,T)}{d(T)} \label{mod_lam_len_eq}
\end{equation}
\end{defn}

\begin{lemma}\label{lam_len_well_defined}
The limit in equation \ref{lam_len_eq} exists, and it is independent
of the choice of $h$. Moreover, the map $pl_{T_0}(\Lam^+,\cdot):
\overline{\os} \to \RR$ is continuous.
\end{lemma}

\begin{proof}
We begin by showing that the limit exists. This boils down to the
fact that $\lam$ is quasi-periodic. We first give the idea of the proof: If $\sigma \subseteq \lam$ is
long enough then $\sigma$ is a concatenation of a list of words $\tau^k_1, \dots , \tau^k_m$ called $k$-tiles, which are $f^k$ images of edges. The $k$-tiles appear with fixed frequencies $r_1, \dots , r_m$. We can choose the tiles long enough so that the cancellation in $h_\#(\wt{\tau^k_i}) h_\#(\wt{\tau^k_j})$ is negligible with respect to the length of
$\tau^k_i$. Thus $h_\#(\wt{\sigma})$ (up to small cancellation) is a concatenation of the 
tiles $h_\#(\wt{\tau^k_i})$, which appear with frequency $r_i$. Thus 
$\displaystyle \frac{l_{T_0}(h(\wt{\sigma})),T)}{l_{T_0}(\wt{\sigma},T_0)} \sim
\frac{\sum_{i=1}^{m}r_i l(h(\wt{\tau^k_i}),T)}{\sum_{i=1}^{m}r_i
l(\wt{\tau^k_i},T_0)}$. This expression can easily be shown to converge
as $k \to \infty$. \\

Let $L = \lip(h)$ and $C = BCC(h)$. Denote the edges of $G_0$ by
$e_1, \dots , e_m$. For each $k$ the $i$-th $k$-tile is $\tau^k_i =
f^k(e_i)$ where $1 \leq i \leq m$. 
We use $l^k_i(T_0) =
l(\wt{\tau^k_i},T_0)$, and $l^k_i(T) =
l([h(\wt{\tau^k_i})],T)$ for shorthand. Let $A = \max\{ l^k_i(T_0) | 1 \leq i \leq m \}$ and $B = \min \{ l^k_i(T_0)
| 1 \leq i \leq m \}$. Suppose $k$ is large enough so that
$\frac{2C}{B}< \eps$. 

Each leaf
$\lam$ of $\Lam^+_0$ has a natural $1$-tiling by edges in $G_0$. The
standard $j$-tiling of $\lam$ is the $f^j$ image of the $1$-tiling of
$f^{-j}(\lam)$. Let $\sigma$ be a subsegment of $\lam$, $\sigma$ itself might not be "nicely" tiled because it might begin and end in the middle of a tile, but we can sandwich it $\sig_1 \subseteq \sig \subseteq \sig_2$ with leaf segments which are tiled. Let $\sigma_1$ be the longest subsegment of $\sigma$ which is tiled by $\{ \tau^k_i \}_{i=1}^m$ and $\sig_2$ the shortest subsegment of $\lam$ which contains $\sig$ and is tiled by $\{ \tau^k_i \}_{i=1}^m$. 
\[ l(\wt{\sigma_1},T_0) \leq l(\wt{\sig},T_0) \leq l(\wt{\sigma_2},T_0) \leq
l(\wt{\sigma_1},T_0) + 2A\]

Let $N^k_i = \# \text{occurrences of the tile } \tau^k_i \text{ in the
tiling of } \sigma_1$, and $N^k = \sum_{i=1}^{m}N_i$. By
Perron-Frobenius theory there are $r_1, \dots , r_m$ (independent of $k$) such that $\displaystyle \frac{N^k_i}{N^k} \to r_i$ as $\sig_1 \to \lam$. Let $a_k = \displaystyle \frac{\sum_{i=1}^{m}r_il^k_i(T)}{\sum_{i=1}^{m}r_i l^k_i(T_0)}$. We show that for large enough $\sigma$, $\displaystyle \frac{l(h[\wt{\sigma}],T)}{l(\sigma,T_0)}$ is in $[a_k-\eps, a_k+\eps]$.\\

We have:
 \[ \frac{l([h(\ws_1)],T)}{l(\ws_2,T_0)} \leq
    \frac{l([h(\ws)],T)}{l(\ws,T_0)} \leq
    \frac{l([h(\ws_2)],T)}{l(\ws_1,T_0)}\] 
The right hand side:
\[
\begin{array}{l}
\displaystyle \frac{l([h(\ws_2)],T)}{l(\ws_1,T_0)} \leq
\frac{\sum_{i=1}^{m} N^k_i l^k_i(T) + 2AL }{\sum_{i=1}^{m}N^k_il^k_i(T_0)} = 
\displaystyle \frac{\sum_{i=1}^{m} \frac{N^k_i}{N^k}l^k_i(T) + \frac{2AL}{N^k}}
{\sum_{i=1}^{m}\frac{N^k_i}{N^k}l^k_i(T_0)} \xrightarrow[N^k \to
\infty]{\phantom{blah}} \frac{\sum_{i=1}^{m}r_i
  l^k_i(T)}{\sum_{i=1}^{m}r_i l^k_i(T_0)} = a_k \end{array}
 \] 
  The left hand side limits to:
\[\begin{array}{lcl}
\displaystyle 
\frac{l(h[\ws_1],T)}{l(\ws_2,T_0)} & \geq & 
\displaystyle 
\frac{\sum_{i=1}^{m} N^k_i \left[l^k_i(T) -2C \right]}{\sum_{i=1}^{m}N^k_i l^k_i(T_0) + 2A}  
 = \frac{\sum_{i=1}^{m}N^k_i l^k_i(T)}{\sum_{i=1}^{m}N^k_i
  l^k_i(T_0) + 2A}
  -\frac{\sum_{i=1}^{m}N^k_i2C}{\sum_{i=1}^{m}N^k_i l^k_i(T_0)
  + 2A} \\[0.5 cm] 
 \displaystyle 
  & \geq & 
  \displaystyle \frac{\sum_{i=1}^{m}\frac{N^k_i}{N^k}
    l^k_i(T)}{\sum_{i=1}^{m}\frac{N^k_i}{N^k}
  l^k_i(T_0) + 2A} - \frac{N^k \cdot 2C}{N^kB+2A}  
  \xrightarrow[N^k \to \infty]{\phantom{blah}} 
  \frac{\sum_{i=1}^{m}r_i
  l^k_i(T)}{\sum_{i=1}^{m}r_i l^k_i(T_0)} - \frac{2C}{B} \\[0.7 cm]
  & = & \displaystyle 
  \frac{\sum_{i=1}^{m} r_i l^k_i(T)}{\sum_{i=1}^{m}r_i l^k_i(T_0)} - \eps = a_k -\eps
  \end{array}\]
Thus, for all $\eps$, and
for large enough $\sigma$:
\begin{equation}\label{sandwich_eq}
a_k - \eps \leq
\frac{l([h(\ws)],T)}{l(\ws,T_0)} \leq a_k + \eps \end{equation}

Let $J(L)$ be the smallest closed interval containing $\left\{ \left. \displaystyle \frac{l([h(\wt\sig)], T)}{l(\wt\sig, T_0)} ~ \right| ~ \sig \subseteq \lam , l(\wt\sig, T) \geq L \right\}$. $J(L+1) \subseteq J(L)$. By equation \ref{sandwich_eq} the diameter of $J(L)$ is bounded by $2\eps$. By Cantor's nested intervals lemma $a_k$ converges to a limit $c$. Thus  
\begin{equation} \displaystyle \lim_{\sig \to \lam}
\frac{l(h[\ws],T)}{l(\ws,T_0)} = c  \label{limit_eq} \end{equation}\\

Next, we show that this limit does not depend on the choice of $h$. We
claim that if $h':T_0 \to T$ is another equivariant Lipschitz map,
then $|l([h(\sig)],T)-l([h'(\sig)],T)|<2D$ for some $D$. Thus the limit in equation
\ref{limit_eq} is the same for both $h$ and $h'$. Indeed let $p$ be
some point in $T_0$. Then for all $x \in T_0$ there is a $g \in F_n$ such
that $d_{T_0}(x,g \cdot p) \leq 1$. Hence $d(h(x),h'(x)) \leq
d(h(x),h(gp)) + d(h(gp),h'(gp)) + d(h'(gp),h'(x)) \leq \text{Lip}(h)
+ d(h(p),h'(p)) + \text{Lip}(h')$. Denote this constant by $D$. Thus,
for any path $\sigma \subseteq T_0$ the initial and terminal
endpoints of $h(\sig), h'(\sig)$ are $D$-close, so
$|l([h(\sig)],T)-l([h'(\sig)],T)|<2D$. \\

Finally we want to show that $pl_{[T_0]}(\Lam^+,[T])$ depends
continuously on $[T]$.  \[pl_{[T_0]}(\Lam^+,[T]) = \lim_{k \to \infty}
 \frac{1}{\sum_{i=1}^{m}r_il^k_i(T_0)} \frac{\sum_{i=1}^{m}r_i
  l^k_i(T)}{d(T)}\] Without loss of generality  suppose $\tr(w_1,T) \neq 0$.
If $[T_j]  \xrightarrow[j \to \infty]{\phantom{blah}} [T]$ then
$\displaystyle \frac{l^k_i(T_j)}{\tr(w_1,T_j)} \to \frac{l^k_i(T)}{\tr(w_1,T)}$ for all $1 \leq i \leq m$ so:
\[ \frac{\sum_{i=1}^{m} r_i l^k_i(T_j)}{d(T_j)} = \frac{\sum_{i=1}^{m} r_i
l^k_i(T_j)/\tr(w_1,T_j)}{d(T_j)/\tr(w_1,T_j)}  \xrightarrow[j \to
\infty]{\phantom{blah}}  \frac{\sum_{i=1}^{m}
r_i l^k_i(T)/\tr(w_1,T)}{d(T)/\tr(w_1,T)} = 
 \frac{\sum_{i=1}^{m}r_il^k_i(T)}{d(T)} \]
 \end{proof}
 
 \subsection{The Whitehead Graph of the Attracting 	
 and Repelling Laminations}

We've defined the Whitehead graph of a conjugacy class $\alpha$ in the basis $\mathcal{B}$. If $R \in \os$ is a rose, i.e. a wedge of $n$ circles, then a marking inverse identifies it's edges with a basis $\BB(R)$ of $F_n$. The Whitehead graph of $\al$ in $R$ is $Wh_R(\al) = Wh_{\BB(R)}(\al)$.\\ 
\indent Let $\eta$ be a  quasi-periodic bi-infinite edge path in $R$. Since it is quasi-periodic, there is some $M$ so that $\{ \text{the turns taken by } \eta\} \subseteq \{ \text{the turns taken by } \eta' \}$ where $\eta'$ is any subpath of $\eta$ whose length is at least $M$. We may assume that $\eta'$ is a closed path. Let $Wh_R^*(\eta')$ be the Whitehead graph of $\eta'$ taking into account all turns except the one connecting the end of $\eta'$ and the beginning of $\eta'$. Define $Wh_R(\eta) = Wh_R^*(\eta')$ (We exclude the "last" turn of $\eta'$ since it need not be taken in $\eta$).\\

\begin{lemma}\label{no_cut_ver_lemma}
There is a point $F  \in \os$ such that for any leaves $\lam \in
\Lam^+_\phi(F)$ and $\nu \in \Lam^-_\phi(F)$, the whitehead graph
$Wh_F(\lam,\nu)$ is connected and contains no cut vertex.
\end{lemma}

To prove this lemma we will need the following proposition proven by
Levitt and Lustig \cite{LL}.

\begin{prop}\label{LLprop}
If $pl_{T_0}(\Lam^+,[T]) = 0$ then $pl_{T_0}(\Lam^-, [T]) \neq 0$
\end{prop}

\begin{proof}
Proposition 5.1 in \cite{LL} shows this for a tree $T$ with dense
orbits. For a general tree the proof can be found in Section 6 of
\cite{LL}.
\end{proof}

\begin{proof}[Proof of Lemma \ref{no_cut_ver_lemma}]
First recall that if $\lam_1, \lam_2 \in \Lam^+(X)$ are leaves of the
attracting lamination then they share the same leaf segments so for
any $X \in \os$,
$Wh_X(\lam_1)= Wh_X(\lam_2)$. Since the choice of the leaves does not
affect the whitehead graph, fix leaves $\lam \in \Lam^+$ and $\nu \in
\Lam^-$ once and for all.

Pick a point $X_0 \in \os$ whose underlying graph is a rose where all
edges have length $\frac{1}{n}$. It was
proven in \cite{BFH} that $Wh_{X_0}(\nu), Wh_{X_0}(\lam)$ are both
connected. If $Wh_{X_0}(\nu, \lam)$ contains a cut
vertex $a$ then let $X_1 = X_0 \cdot \phi_{(A,a)}$ the automorphism described in Whitehead's algorithm (see description in the paragraph following Definition \ref{free_factors}). Continue this way to
get a sequence $X_0, X_1, X_2 , \dots $ We will show that this
process terminates in a finite number of steps with a graph $F = X_N$ such that $Wh_{F}(\nu, \lam)$ does not contain a cut vertex. \\

A priori, two other cases are possible: $X_k = X_j$ for some $j>k$, and
the process never terminates producing an infinite sequence $\{ X_i
\}_{i=1}^{\infty}$.

\begin{obs}\label{main_obs}
For all $i$ we have $pl_{T_0}(\Lam^+,\wt{X_{i}}) >
pl_{T_0}(\Lam^+,\wt{X_{i+1}})$ and  $pl_{T_0}(\Lam^-,\wt{X_{i}}) >
pl_{T_0}(\Lam^-,\wt{X_{i+1}})$
\end{obs}

We delay the proof of this observation to finish the proof of Lemma \ref{no_cut_ver_lemma}. $X_k = X_j$ for $k<j$  is impossible since
the lengths get strictly smaller. If the process doesn't terminate then we get an infinite sequence $\{ X_i \}_{i=1}^{\infty}$ which has a subsequence
converging to $[T] \in \overline{\os}$. Since the $X_i$s are part of an orbit, and $\out$ acts discretely on $\os$, the limit $\lim_{i \to \infty} X_i$ must lie in $\partial \os$. We will argue that $pl_{T_0}(\Lam^+,[T]) = pl_{T_0}(\Lam^-,[T]) =0$  and get a contradiction to Proposition \ref{LLprop}. \\

Let $L =
l_{T_0}(\Lam^+,\wt{X_0})$
then $l_{T_0}(\Lam^+, \wt{X_i}) < L$, and together with $d(\wt{X_i}) \geq 1$ we get $pl_{T_0}(\Lam^+,[\wt{X_i}]) < L$. Therefore,
$pl_{T_0}(\Lam^+,[T]) < L$. Now assume by way of contradiction that
$pl_{T_0}(\Lam^+,[T]) = L' >0$. There exists some conjugacy class $[w]$ such that $\tr(w,T)<
\frac{L'}{2nL}$ (if $T$ is simplicial then there is a
conjugacy class $[w]$ which
is elliptic and if $T$ is not simplicial, it has a quotient tree with 
dense orbits. In either case we can find conjugacy classes with arbitrarily
small translation length). Since $\wt{X_k}$ converges projectively to
$[T]$,
\[ \frac{tr(w,\wt{X_k})}{d(\wt{X_k})} \to \frac{tr(w,T)}{d(T)} <
\frac{1}{d(T)} \frac{L'}{2nL} < \frac{L'}{2nL}\]
Thus, for a large enough $k$,
$\displaystyle \frac{\tr(w,\wt{X_k})}{d(\wt{X_k})} <
\frac{L'}{nL}$ which implies
\[ \frac{\tr(w,\wt{X_k})}{l_{T_0}(\Lam^+,\wt{X_k})} = \frac{\tr(w,\wt{X_k})}{d(\wt{X_k})} \frac{d(\wt{X_k})}{l_{T_0}(\Lam^+,\wt X_k)}< \frac{L'}{nL}\cdot \frac{1}{L'} = \frac{1}{nL} \]
But $l_{T_0}(\Lam^+, \wt{X_k})<L$, and
$\displaystyle \tr(w,\wt{X_k}) \geq \frac{1}{n}$ so $\displaystyle \frac{\tr(w,\wt{X_k})}{l(\Lam^+, X_k)} > \frac{1}{nL}$.
So we get a contradiction to $pl_{T_0}(\Lam^+,[T]) \neq 0$. A similar
argument shows $pl_{[T_0]}(\Lam^-,[T]) = 0$ and we get a contradiction. Therefore,
the process must end in a finite number of steps with a graph $F$ such
that $Wh_F(\lam,\nu)$ is connected without a cut vertex.
\end{proof}

\begin{remark}
Experimental evidence suggests that one can actually choose $F$ to lie
on an axis of $\phi$, but we were not able to show that.
\end{remark}

\begin{proof}[Proof of Observation \ref{main_obs}]
Let $a_k(T) = \frac{\sum_{i=1}^m r_i l_i(T)}{\sum_{i=1}^m r_i
l_i(T_0)}$ where the notation is established in the proof of Proposition \ref{lam_len_well_defined}. We must estimate $\displaystyle \lim_{k \to \infty}a_k(T)$ for $T = \wt{X}_i$ and
$T = \wt{X}_{i+1}$. We will show that  $l_i(\wt X_k) > l_i(\wt X_{k+1}) $. \\

Recall that $Wh_{X_k}(\lam)$ is connected and contains a cut vertex $a$ and that $X_{k+1} = X_k \cdot \phi_{(A,a)}$ (here we do not distinguish between the vertices of the Whitehead graph and the directed edges of the rose $X_k$). Let $h:G_0 \to X_i$ be a difference in marking Lipschitz map. $C=BCC(h)$ and $M = \max \{ nC , 1 \}$. Let $\tau_i$ be $k$-tiles in $G_0$, with $k$ large enough so that $[h(\tau_i)]$ contains a closed subpath $\gamma$ of $\lam(X_i)$ which contains at least $5M$ turns of the form $\bar{x}a$ and $\bar{a}x$ with $\bar{x} \in A$. Notice that by Definition \ref{Wh_move} these are precisely the turns where cancellation occurs after applying the Whitehead automorphism. Thus \[ l(\phi_{(A,a)}(\gamma), X_{j+1}) < l(\gamma, X_j) - 5M \cdot \frac{1}{n} < l(\gamma,X_j) - 5C \]
Since $[h(\tau_i)]$ is contained in $\lam(X_i)$ except for subsegments of length at most $C$ at both ends. These segments contain at most $Cn$ edges. They might become longer under $\phi_{(A,a)}$. But we can estimate:
\[ l(\tau_i, X_{j+1}) < l(\tau_i, X_j) -  5C + 2 \cdot Cn \cdot \frac{2}{n} < l(\tau_i, X_j) - C\] 
Thus there is a $C'$ such that $a_k(\wt X_j) > a_k(\wt X_{j+1}) + C'$ and this holds in the limit as well.
\end{proof}

\begin{remark}\label{no_cut_ver_rmk}
If $Wh_F(\lam,\nu)$ is connected and does not contain a cut vertex, then $Wh_{F
\cdot \phi}(\lam,\nu)$ and $Wh_{F \cdot \phi^{-1}}(\lam,\nu)$ satisfy 
the same property. In fact $Wh_F(\lam,\nu)= Wh_{F \cdot
  \phi}(\lam,\nu) = Wh_{F \cdot \phi^{-1}}(\lam,\nu)$. Indeed let $k:F
\to F$ and $k':F \to F$ be topological
representatives of $\phi, \phi^{-1}$ (see Definition \ref{topo_repn}). Then $Wh_F(\lam) = Wh_{F \cdot \phi }(k_\#(\lam))$ this is
because $\lam \in \Lam^+(F)$ implies $k_\#(\lam) \in \Lam^+(F \cdot
\phi)$. $Wh_{F \cdot \phi }(k_\#(\lam)) = Wh_{F \cdot \phi}(\lam)$
because $\Lam^+(F \cdot \phi)$ is $\phi, \phi^{-1}$-invariant
. Thus $Wh_F(\lam) = Wh_{F\phi}(\lam)$. Similarly, $Wh_{F}(\nu) = Wh_{F \cdot
\phi}(k_\#(\nu)) = Wh_{F \cdot \phi}(\nu)$. Thus, $Wh_F(\lam,\nu)= Wh_{F \cdot
\phi}(\lam,\nu)$. The argument for $Wh_{F \phi^{-1}}(\lam,\nu)$ is
identical.
\end{remark}


\section{Axes are strongly contracting}

\subsection{Projections of Horoballs are Finite}

\begin{definition}
Let $\eta$ a leaf of $\Lam^+$ or $\Lam^-$ in $X$. Let $\gamma$ be an edge path contained in $\eta$. We say that $\gamma$ is an \emph{$r$-piece} of $\eta$ if the $l(\gamma,X) \geq r$.
\end{definition}

The next proposition states that basis elements cannot contain long
pieces of both $\Lam^+$ and $\Lam^-$.

\begin{prop}\label{basis_prop}
There exists a constant $\jmath>0$ so that for all $G_t \in \LL_f$:
\begin{enumerate}
    \item Let $\beta$ conjugacy class of $F_n$. $\beta$ is represented by an immersed loop which we shall also denote by $\beta$ in $G_t$. Suppose there exist leaves $\lam \in \Lam^+_f(G_t)$
    and $\nu \in \Lam^-_f(G_t)$ such that $\beta$ contains a $\jmath$-piece
    of $\lam$ or the inverse of a $\jmath$-piece of $\lam$ and a $\jmath$-piece of $\nu$. Then $\beta$ is not a basis element.
    \item Let $\al, \beta$ be tight loops in $G_t$ (also thought of as conjugacy classes). Suppose $\al,\beta$ are compatible with a free
    decomposition of $F_n$. If
    $\al$ contains a $\jmath$-piece of $\lam$ or the inverse of a $\jmath$-piece of $\lam$ (a $\jmath$-piece of $\nu$ or the inverse of a $\jmath$-piece of $\nu$) then $\beta$ doesn't
    contain a $\jmath$-piece of $\nu$ (a $\jmath$-piece of $\lam$ or the inverse of a $\jmath$-piece of $\lam$).
\end{enumerate}
\end{prop}

\begin{proof}
\begin{enumerate}
\item \label{A}
We first prove this for $G_0$. By lemma 
\ref{no_cut_ver_lemma}, there is an $F \in CV_n$ such that
$Wh_F(\lam,\nu)$ is connected and contains no cut point. Suppose $d =
d(F,G_0)$ and $k = \exp(d)$ so 
for all loops $\al$: $\frac{l(\al,G_0)}{l(\al,F)} \leq k$. Hence $l(\al, F) \geq 
\frac{1}{k} l(\al,G_0)$. Let $h_0: G_0 \to F$ be an optimal Lipschitz
 difference in marking, and let $C= BCC(h_0)$. Since $\lam_F, \nu_F$ are
quasi-periodic there is a length $r$ such that if $\gamma_F$ is an
$r$-piece of $\lam_F$ then $\gamma_F$ contains all of the $2$-edge leaf
segments in $\lam_F$ hence $Wh_F(\lam_F) = Wh_F^*(\gamma_F)$. Similarly, if $\delta_F$ is an $r$-piece of $\nu_F$ then $\delta_F$ contains all of  the $2$-edge leaf segments in $\nu_F$ hence  $Wh_F(\nu_F) =
Wh_F^*(\delta_F)$. Since $\lam,\nu$ are quasi-periodic there is a length $B$ so that all edges appear in any $B$-piece of $\lam$ or $\nu$. \\
$\phantom{blah}$ Let $\jmath = k(r+2C)+B$. If $\beta_{G_0}$ contains $\beta_1'$ or $\beta_1^{'-1}$ where $\beta_1'$ is a $\jmath$-piece of  $\lam_0 \in \Lam^+(G_0)$. By truncating a piece of length at most $B$ from $\beta_1'$ we can find $\beta_1 \subseteq \beta'_1$ a loop such that $l(\beta_1,G_0) > k(r+2C)$. Thus $l([h_0(\beta_1)],F)>r+2C$ and $[h_0(\beta_1)]$ is contained in $\lam_F$ apart from some initial and terminal segments of length at most $C$. Hence $[h_0(\beta_{G_0})]$ contains an
$r$-piece of $\lam_F$. Similarly, if $\beta_{G_0}$ contains an
$\jmath$-piece of $\nu_0 \in \Lam^-(G_0)$ then there is a loop $\beta_2 \in \nu$ so that $l(\beta_2,G_0)>k(r+2C)$. Hence 
$l([h_0(\beta_2)],F)>r+2C$ and $[h_0(\beta_{G_0})]$ contains an
$r$-piece of $\nu_F$. Therefore, if $\beta_{G_0}$ contains such $\beta'_1, \beta'_2$ (see Figure \ref{long_pieces_fig}) 
then $Wh_F(\beta) \supseteq Wh_F(\lam, \nu)$. By the definition of the Whitehead graph $Wh_{F}(\gamma_F) = Wh_{F}(\gamma_F^{-1})$ so if $\beta$ contains $\beta_1^{'-1}, \beta_2'$ then again,  $Wh_F(\beta) \supseteq Wh_F(\lam, \nu)$. Thus $Wh_F(\beta)$ is connected and does not contain a cut vertex. By Whitehead's theorem \ref{WhThm} $\beta$ is not a basis element.

\begin{figure}[ht]
\begin{center}
\input{long_pieces.pstex_t}
\caption{\label{long_pieces_fig}A basis element cannot contain long pieces of both laminations}
\end{center}
\end{figure}
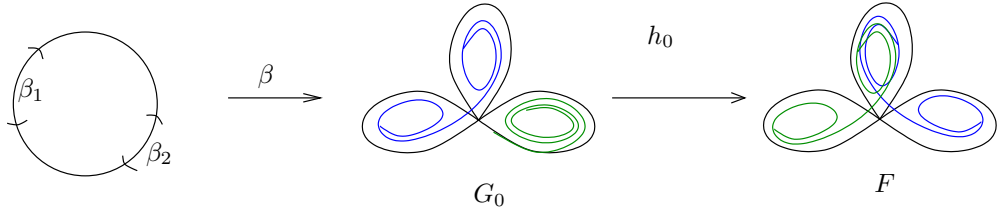

$\phantom{blah}$ We can do the same for all graphs $G_t \in \LL_f$ and $\jmath$ depends
on $d(F,G_t)$, which varies continuously 
with $t$. Therefore if we
vary $t$ across a fundamental domain of the $\phi$ action on $\LL_f$, then there is an upper bound for $\jmath$ (which we still denote $\jmath$). Now
by Remark \ref{no_cut_ver_rmk} the same is true (with the same
$\jmath$) for any translate of the fundamental domain (we translate
$F$ as well so the distance and the optimal map remain the same).

\item The proof of the second claim is similar to \ref{A}. If $\al_G$ contains a long enough piece of $\lam_G$ and $\beta_G$ contains a long enough piece of $\nu_G$ then $Wh_{F}(\al, \beta) \supseteq  Wh_F(\lam_F, \nu_F)$ but by Theorem \ref{WhVer2} $\al,\beta$ are then not compatible with a free decomposition.
\end{enumerate}
\end{proof}

We now turn to prove some applications: 

\begin{lemma}\label{close_min_for_factors}
There is an $s>0$ such that: if $\al, \beta$ are conjugacy classes
which are compatible with a free decomposition of $F_n$ then
$|t_{\al}-t_{\beta}|<s$
\end{lemma}

\begin{proof}
Suppose ${t_\beta}>{t_\al}$. Let
$\al_t$ represent $\al$ in $G_t$, and $\beta_t$ represent $\beta$ in
$G_t$. We claim that there is a $t_0$ such that if $t<{t_\beta}-t_0$ then
$\beta_t$ contains a $\jmath$-piece of $\nu_{G(t)}$, and if
$t>{t_\al}+t_0$ then $\al_t$ contains a $\jmath$-piece of $\lam_{G(t)}$.
Thus, if $|{t_\beta}-{t_\al}|>2t_0$ let $r={t_\al}+t_0$ then $\al_r$ contains an
$\jmath$-piece of $\lam_{G(r)}$ and $\beta_r$ contains a $\jmath$
piece of $\nu_{G(r)}$ which contradicts proposition
\ref{basis_prop}. 

To find $t_0$: by proposition \ref{pos_leg}, there is an
$s_1=s_{\ref{pos_leg}}$ such that if $t>{t_\al}+s_1$ then $LEG_f(\al_t,
G(t))> \eps_0$. Let $\al'_t \subseteq \al_t$ be a legal segment of length
$>\kappa$ (the legality threshold). There is an $N$ such that $f^N(\al'_t)$ is longer 
than $\frac{2\jmath}{\text{PF}(f)+1}$ here $\text{PF}(f)$ is the Perron-Frobenius eigenvalue of $f$ (which we denoted earlier by $\lam$,  however, now we use $\lam$ to denote a periodic leaf in $\Lam^+$). By the argument in the paragraph just before definition \ref{def_legality}, $[f_t^N(\al_t)]$ will contain a $\jmath$-piece of the lamination contributed from $f_t^N(\al'_t)$. Let $s_2=s_1+N \log(\text{PF}(f))$ then at
$t_0={t_\al}+s_2$, $\al$ contains a $\jmath$-piece of $\lam$, contributed
by $\al'_t$. Similarly for
$g$, the result follows from the fact that $\LL_f$ and $\LL_g$ are
close, and from the fact that $t_{\al}$ and $t'_{\al}$ are close (by
corollary \ref{close_minima}).
\end{proof}

\begin{cor}\label{close_can_cor}
There exists a constant $s>0$ such that if $\alpha$ and $\beta$ are
candidates in $X$ then
$|t_\alpha-t_\beta|<s$
\end{cor}

\begin{proof}
By Proposition \ref{can_are_basis} there is a candidate $\gamma$ so that $\al, \gamma$ and $\gamma, \beta$ can be completed to a basis of $F_n$. Therefore by lemma \ref{close_min_for_factors} there is an $s=s_{\ref{close_min_for_factors}}$ such that $|t_\al - t_\gamma|< s$ and $|t_\gamma - t_\beta|<s$. Thus $|t_\al-t_\beta|<2s$
\end{proof}

\begin{cor}\label{min_can_is_proj}
There exists a constant $s>0$ so that if $\al$ is a candidate in $X$ then $|t_X - t_\al|<s$
\end{cor}
\begin{proof}
Let $\al_1, \dots , \al_N$ be the candidates of $X$. Then for each $i$, $\displaystyle \min_{t \in \RR} \ST_{\al_i}(X,G(t)) = \ST_{\al_i}(X, G(t_{\al_i}))$. By the proof of Proposition \ref{can_are_basis} the minimum of \[ h(t) = \max\{ \ST_{\al_i}(X,G(t)), \ST_{\al_j}(X,G(t)) \}\] is realized by a point in $[ \min\{ t_{\al_i}, t_{\al_j} \}, \max\{ t_{\al_i}, t_{\al_j} \} ]$. Thus (by induction) the minimum of $d(X, G(t)) = \max \{\ST_{\al_i}(X,G(t)) \mid 1 \leq i \leq N \}$ is realized at $t = t_X$ in $[ \min \{ t_{\al_i}\mid 1 \leq i \leq N \}, \max \{ t_{\al_i}\mid 1 \leq i \leq N \} ]$. By Corollary \ref{close_can_cor} the length of this interval is $s = s_{\ref{close_can_cor}}$. Thus $|t_X - t_\al|<s$.
\end{proof}

\begin{cor}\label{short_loops}
There exists an $s>0$ such that if the translation length of $\al \in
F_n$ in both $X$ and $Y$ is smaller than $1$ then
$|\pi(X)-\pi(Y)|<s$.
\end{cor}
\begin{proof}
Let $<x_1, x_2, \dots, x_n>$ be a short basis for $\pi_1(X)$, and $<y_1, y_2, \dots, y_n>$ a short basis for $\pi_1(Y)$ (all loops are
smaller than $1$). Since $vol(X) = 1$, $\al$ is carried by a free factor: $<x_1, \dots , x_k>$.
So $|t_{\al}-t_{[x_n]}|<s_{\ref{close_min_for_factors}}$. Similarly,
for $Y$,$|t_{\al}-t_{[y_n]}|<s_{\ref{close_min_for_factors}}$. So
$t_{[x_n]}$ and $t_{[y_n]}$ are uniformly close. By corollary
\ref{min_can_is_proj}, we have that $t_X$ and $t_Y$ are uniformly
close.
\end{proof}

A horoball based at the conjugacy class $\al$ is the (unbounded) subset $H(\al,r) = \{x \in \os \mid l(\al,x) < r \}$ of $\os$. 
Corollary \ref{short_loops} shows that the $\pi_f(H(\al,1))$ is a bounded interval of $\LL_f$. 

\subsection{Projections to Axes are Like Projections in Trees} Consider a geodesic $\LL$ in a tree $T$ , and let $\pi:T \to \LL$ be the
closest point projection. The next lemma is motivated by the
following observation (see Figure \ref{tree_prop1_fig}): If $X$ is a point on $\LL$ then $d(Y,X) = d(Y,
\pi(Y)) + d(\pi(Y),X)$.

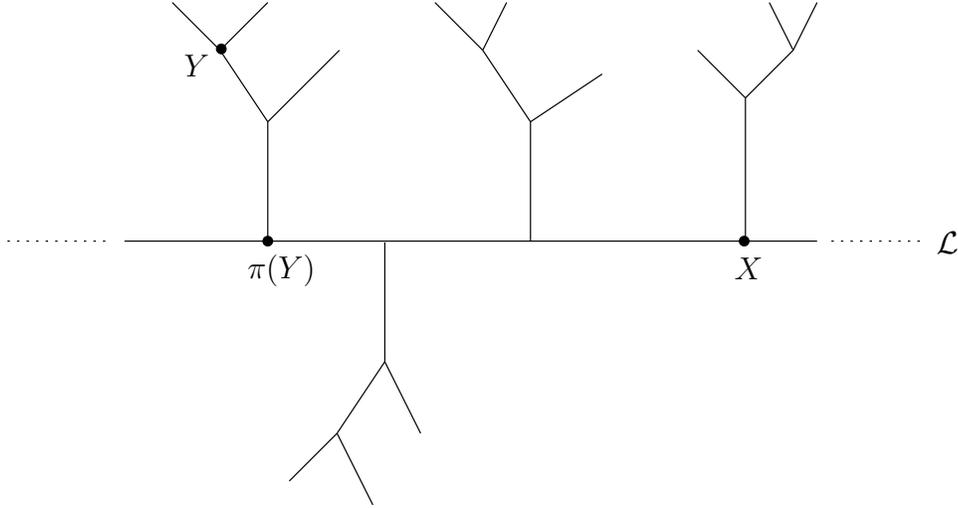
\begin{figure}[b]
\begin{center}
\input{tree_prop1.pstex_t}
\caption{\label{tree_prop1_fig}In a tree, the geodesic from $Y$ to a point on a geodesic
  visits $\pi(Y)$}
\end{center}
\end{figure}

\begin{lemma}\label{tree_like_prop1}
There exist constants $s,c>0$ such that for any $Y$, if $|t - t_Y|>s$
then $d(Y,G(t)) \geq d(Y,\pi(Y))+d(\pi(Y),G(t))-c$
\end{lemma}
\begin{proof}
Denote $X=G(t)$. Let us first prove it for $t>t_Y$. There is an $s_1
= s_{\ref{min_can_is_proj}}$ such that  for all
candidates $\al$ of $Y$: $|t_\al - t_Y|<s_1$. There is an $s_2=
s_{\ref{pos_leg}}$ such 
that if $t>t_\al+s_2$ then $LEG_f(\al_t, G(t))>\eps_0$. Let $Z =
G(t_Y+s_1+s_2)$ then for any candidate $\beta$ of $Y$,
$LEG_f(\beta, Z)> \eps_0$. Now suppose $\beta_Y$ in $Y$ is the loop
that realizes $d(Y,Z)$, i.e. $\ST_\beta(Y,Z) = e^{d(Y,Z)}$. Then,
since $\beta$ is $\eps_0$-legal in $Z$ then by Corollary \ref{alm_legal_almost_max_st} there is a $C = C_{\ref{alm_legal_almost_max_st}}$ so that 
 $\ST_\beta(Z,X) \geq C e^{d(X,Z)}$ so $\ST_\beta(Y,X) = \ST_\beta(Y,Z) \ST_\beta(Z,X)
\geq C e^{d(Y,Z)}e^{d(Z,X)} = C e^{d(Y,Z)+d(Z,X)}$. We have 
$\ST(Y,X) \geq \ST_\beta(Y,X) \geq C e^{d(Y,Z)+d(Z,X)}$. Thus
$d(Y,X) \geq \log(C) + d(Y,Z) + d(Z,X)$. Now recall that $Z =
G(t_Y+ s_1+s_2)$ so $d(\pi(Y), Z) = s_1+s_2$. We have,
\[ \begin{array}{l}
d(Y, Z) > d(Y,\pi(Y)) \\
d(Z, X) > d(\pi(Y),X) - (s_1 + s_2)
\end{array}\] thus 
$d(Y,X) \geq d(Y,\pi(Y)) + d(\pi(Y),X) - (s_1+s_2) + \log(C)$
let $c = s_1+s_2 - \log(C)$ and we get $d(Y,X) \geq d(Y,\pi(Y))+
d(\pi(Y), X) - c$.

If $t<t_Y$: there is an $s'$ such that the above holds for $g$. The
claim now follows form the fact that $\pi_f, \pi_g$ are uniformly
close (see lemma \ref{close_minima}).

\end{proof}

Getting back to the tree $T$, if $X,Y$ are any two points such that
$\pi(Y) \neq \pi(X)$ then the geodesic from $Y$ to $X$ passes through
$\pi(X)$, see Figure \ref{tree_prop2_fig}. In particular $d(Y,X)> d(Y,\pi(X))$. In $\os$:

\begin{lemma}\label{tree_like_prop2}
There exist constants $s,c >0$ such that for $X,Y \in \os$ if $|t_Y
- t_X|>s$, then $d(Y,X) \geq d(Y,\pi(X)) - c$
\end{lemma}

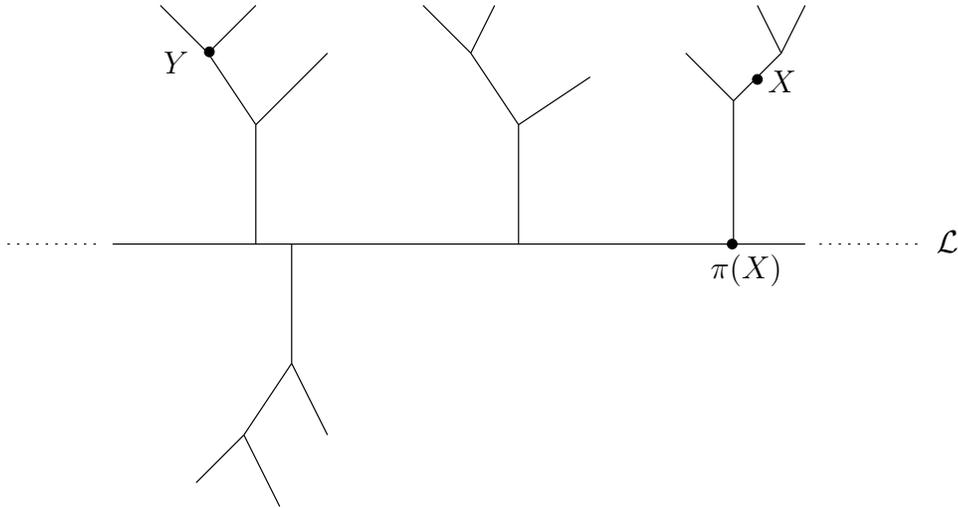
\begin{figure}[b]
\begin{center}
\input{tree_prop2.pstex_t}
\caption{\label{tree_prop2_fig}In a tree, if $X,Y$ project to different points then the
  geodesic between them visits both of the projections.}
\end{center}
\end{figure}

To prove this we recall from lemma \ref{basis_prop}, that if $\al$
and $\beta$ are loops in $G(t)$ representing candidates of $X$ then
they cannot contain long pieces of both laminations $\Lam^+, \Lam^-$. We
will need a slightly souped up version of this.

\begin{definition}
We call a point $X$ in $\os$ minimal if the underlying topological graph of $X$ is either a bouquet of circles or a graph with two vertices, one edge between them which we will refer to as a bar, and all other edges are loops.
\end{definition}

\begin{prop}\label{souped_basis_prop}
Suppose $X$ is minimal. Let $v$ be one of its vertices and the basepoint for $\pi_1(X,v)$ and let $e$ denote the bar of $X$ initiating from $v$ (if $X$ is a rose then $e$ is empty). Let $\al_X, \beta_X$  be either one edge  loops based at $v$ or loops of the form $e \gamma \bar e$ where $\gamma$ is a one edge loop based at the other vertex. Fix $Z \in \LL_f$ and let $h: X \to Z$ be a map homotopic to the difference in marking so that $h(\al_X)$ is an immersed loop and $h(\beta_X)$ is an immersed path. If
$h(\al_X), [h(\beta_X)] $ both contain a $\jmath$-piece of $\nu \in
\Lam^-$ then $h(\beta_X)$ does not contain a $2\jmath$-piece of $\lam \in
\Lam^+$. 
\end{prop}

\begin{proof}
We emphasize that by proposition \ref{basis_prop}, $[h(\beta)]$
does not cross a $\jmath$-piece of $\lam$ but we want it not to contain
any such pieces in the part that gets cancelled when we tighten the
loop. \\
We represent $h(\al)$ by the edge path $x$ in $G(t)$ and
$\beta$ by $u = wyw^{-1}$, with $y$ cyclically reduced. Notice that since $\al^{m}\beta$ represents a basis element for all $m \geq 0$, then $x^{m}u$ represents a basis element for $m \geq 0$. We proceed to prove the proposition by way of contradiction. If $w$ crosses an $2\jmath$-piece of $\lam$ then $w \nsubseteq x^m$, for some $m \geq 1$. For otherwise $x$ would contain a $\jmath$-piece of $\lam$ whence we contradict Proposition  \ref{basis_prop}. 

So there is a large enough $m$ such that when we reduce the path $x^m \cdot wyw^{-1} \cdot x^{m}$ the cancellation happens only at the dots. Write $w= w_1 w_2$ where $w_1$ is the part that is cancelled and $w_2 \neq \emptyset$.\\ 
If $w_2$ contains a $\jmath$-piece of $\lam$. But then $z = [x^m \cdot w_1 w_2 y w^{-1} \cdot x^m]$ represents a basis element and $w_2$ survives after the cancellation. So $z$ will contain a $\jmath$-piece of $\lam$. If $m$ is large enough, $z$ will also contain a copy of $x$. We get a basis element containing $\jmath$-pieces of $\lam$ and $\nu$ thereby contradicting Proposition \ref{basis_prop}. 

Thus $w_1$ contains a $\jmath$-piece of $\lam$. Then $x^m$ contains the inverse of a $\jmath$-piece of $\lam$ and also a $\jmath$-piece of $\nu$ (Here we cannot get an easy contradiction as before since $x$ might not contain a $\jmath$-piece of $\lam^{-1}$). Consider the basis element $u = [x^mx^{m} \cdot w y w^{-1} \cdot x^m ]$.  The first $x^m$ survives after the cancellation and contributes a $\jmath$-piece of $\lam$ and a $\jmath$-piece of $\nu$ to $u$ thereby contradicting Proposition \ref{basis_prop}. 
\end{proof}

\begin{proof}[Proof of Lemma \ref{tree_like_prop2}]
We prove the claim for $X,Y$ such that $t_Y<t_X$, the case where $t_Y>t_X$ follows by applying the same argument to $g$. We also make the assumption that $X$ is minimal.

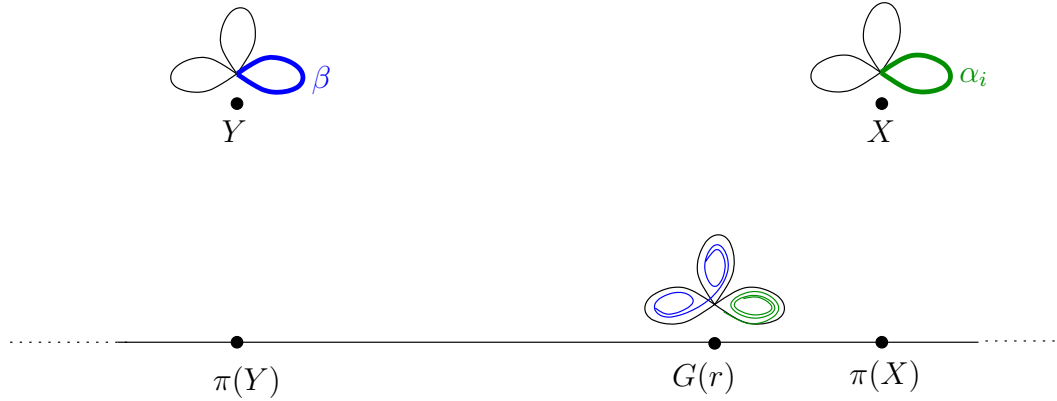
\begin{figure}[b]
\begin{center}
\input{tree_prop2_fig.pstex_t}
\caption{\label{tree_prop_fig}In $G(r)$, $\beta$ contains many $l_1$-pieces of $\lam$ and
  $\al_i$ contain $l_1$-pieces of $\nu$}
\end{center}
\end{figure}

Let $\lam$ be a periodic leaf of $\Lam^+_f$ and $\nu$ a periodic leaf of $\Lam^-_f$. Let $\jmath = \jmath_{\ref{souped_basis_prop}}$. The idea of the proof is as follows. If $t_Y << t_X$, then for $r$ in the middle of $[t_Y,t_X]$, any loop which is short in $Y$, would contain
many $\jmath$-pieces of $\lam$ in $G(r)$. And any loop which is short in
$X$ would contain many $\jmath$-pieces of $\nu$  in $G(r)$, see Figure
\ref{tree_prop_fig}. 
If a candidate in $Y$
was short in $X$, then it would contain pieces of both $\lam$
and $\nu$ in $G(r)$ contradicting the fact that it is a basis element. To make
the argument precise we need to argue that for a candidate $\beta$ in
$Y$ , $l(\beta,X)$ is longer than a definite fraction of
$l(\beta,\pi(X))$. This is done by bounding $l(\beta,\pi(X))$ with the number of
disjoint $\jmath$-pieces of $\lam$ that appear in $\beta_{\pi(X)}$.  

Let $s_1 = s_{\ref{min_can_is_proj}}$ i.e. for
any candidate $\beta$ in $Y$, $|t_Y - t_\beta|<s_1$. Let
$s_2=s_{\ref{pos_leg}}$ i.e. for any primitive conjugacy class $\beta$ if
$t>t_\beta+s_2$ then $LEG_f(\beta, G(t))>\eps_0$.
Let $s_3$ be such that if $t>t_\beta+s_2+s_3$ then $\beta$ crosses
a $\jmath$-piece of $\lam$ in $G(t)$ (contributed by one of the
$\kappa$ long legal segments). Let $s_4$ be such that for
any primitive conjugacy class $\beta$ if $t< t_\beta-s_4$ then $\beta$ contains
a $\jmath$-piece 
of $\nu$ in $G(t)$. Let $s = 2s_1 +s_2 + s_3 + s_4$ and suppose that $t_X - t_Y >s$ we
will show that there exists a $c$ as in the statement of the Lemma.

Let $\beta$ be a loop in $Y$ such that $d(Y, \pi(X)) =
\log(\ST_\beta(Y,\pi(X)))$. Then by proposition \ref{close_can_cor}
$t_\beta< t_Y + s_1$. Let $r = t_X - s_1 - s_4$ then $r> t_Y+s_1 +
s_2+s_3$ (see figure \ref{tree_prop_fig}). Let $k(r)$ be the number of 
$\jmath$-pieces of $\lam$ in $\beta_r \subseteq G(r)$ with disjoint interiors, then
\[ k(r) \cdot \jmath > \eps_0 \cdot l(\beta,G(r)) \]

Recall that $X$ is minimal. Let $\alpha_1, \dots , \alpha_n$ denote the loops representing the short basis in $X$. $\al_i$ is either a one edge loop or is $e \al' \bar e$ where $\al'$ is a one edge loop based at the other vertex and $e$ is the bar of $X$. Let $\al_1$ be the longest one-edge-loop.\\ Choose a map $h:X \to G(r)$, homotopic to the difference in marking, so that $h(\alpha_1)$ is an immersed loop and $h(\al_i)$ are immersed as paths. Each $h(\alpha_i)$ in $G(r)$
contains a $\jmath$-piece of $\nu$. By proposition \ref{souped_basis_prop}  for $1 \leq i \leq n$, $h(\al_i)$ does not contain any $\jmath$-pieces of $\lam$.

\begin{claim}
Let $\gamma$ be a conjugacy class in $F_n$ and write it as a cyclically reduced word in $\al_1, \dots, \al_n$ the basis of $\pi_1(X,v)$. If $[h(\gamma_X)]$ contains $k$ occurrences of  $\jmath$-pieces of $\Lam^+$ in $G(r)$ (with disjoint interiors) then $\gamma_X$ traverses each $\al_q$ at least $k$ times. 
\end{claim}
\begin{proof}[Proof of Claim]
First note that if $\gamma_X$ is a loop that does not traverse
$\al_q$ at all then it is carried by the free factor $<\al_1, \dots , \widehat{\al_q}, \dots, 
\al_n >$. Using proposition \ref{basis_prop} applied to $h(\al_q), 
[h(\gamma)]$ in $G(r)$, we get that $[h(\gamma)]$ does not contain any $l$-pieces of $\Lam^+$ in $G(r)$. \\ 
\indent Now suppose that $\gamma_X = \al_{i_1} \dots
\al_{i_N}$ so that $\al_{i_j}=\al_q$ for at most $k-1$ choices of $j$s.
$[h(\gamma_X)]$ is the result of reducing $h(\al_{i_1}) \cdot
h(\al_{i_2}) \cdots h(\al_{i_N})$ to get $\sig_{i_1} \sig_{i_2} \cdots
\sig_{i_N}$ where $\sig_{i_j}$ are the subpaths of $h(\al_{i_j})$ that survive after the cancellation (some $\sig_{i_j}$ might be
trivial). $\jmath$-pieces of $\lam$ can appear only if they are split between
different $\sigma_i$s. If there are $k$ disjoint $\jmath$-pieces
of $\lam$ in $\gamma_{G(r)}$ then there is a $\jmath$-piece of $\lam$
appearing in $\sig_{i_m} 
\cdots \sig_{i_l} \subseteq [h(\al_{i_m}) \dots h(\al_{i_l})]$ where none of
the $\al_{i_j}$ are equal to $\al_q$. This is a contradiction to the first paragraph. 
\end{proof}

By the claim above $\beta_X$ in $X$ must traverse $\al_1$ at least $k=k(r)$ times. If $l(\al_1,X)>\frac{1}{n+1}$ then $l(\beta,X) > \frac{k}{n+1}$. Otherwise, $X$ has a separating edge  $e$ and $l(e,X)>\frac{1}{n+1}$. Let $\delta$ be a one-edge-loop so that $\al_1$ and $\delta$ are loops on opposite sides of $e$. By the claim above $\beta_X$ traverses $\al_1$ and $\delta$ alternately at least $\frac{k}{2}$ times therefore it must cross $e$ at least $\frac{k}{2}$ times. Again we get $l(\beta,X)> \frac{k}{2(n+1)}$. Therefore, 
\begin{equation}\label{eq_mid}
l(\beta,X) > \frac{\eps_0
}{ 2(n+1) \jmath}l(\beta,G(r))
\end{equation}
$\LL_f$ is contained in the $\theta$-thick part of $\os$ for some $\theta$. Let $\mu = \exp ( D_{\ref{almost_sym}}(\theta) (s_1+s_4))$ then $l(\beta, G(r)) \geq l(\beta, G(t_X)) e^{d(G(t_X), G(r))} \geq \mu l(\beta, G(t_X)) $. By equation \ref{eq_mid} we get  $l(\beta,X)
> \frac{\eps_0 \mu}{2(n+1) \jmath} l(\beta,\pi(X))$. Thus, we get 
$\frac{l(\beta,X)}{l(\beta,Y)} > \frac{\eps_0 \mu}{2(n+1) \jmath}
\frac{l(\beta,\pi(X))}{l(\beta,Y)}$ i.e. \[ d(Y,X) > d(Y,\pi(X)) -
\log \left( \frac{2(n+1) \jmath}{\eps_0 \mu} \right) \] Which proves the statement in the case that $X$ is minimal. Now we deal with the case that $X$ is not minimal.  We claim that there is a constant $b$ such that for any $X \in \os$ there is a minimal $K$ so that $d(X,K)<b$. Moreover, there exists a short loop in $X$ that is still short in $K$. Therefore, by corollary \ref{short_loops}
 $d(\pi(K), \pi(X))<s_{\ref{short_loops}}$. So $d(Y,X) \geq d(Y,K) - d(X,K)  \geq
 d(Y,K) - b > d(Y, \pi(K)) - c -b \geq d(Y, \pi(K)) - d(\pi(K), \pi(X)) -c -b 
> d(Y,\pi(X)) -c -b -s_{\ref{short_loops}}$. \\ To prove that each point in $\os$ lies a uniform distance away from a
minimal $K$:  
Let $e$ be the longest edge in $X$. Note that $l(e,X) \geq \frac{1}{3n-3}$. If $e$ is nonseparating let $J$ be a maximal tree in $X$ that does not contain $e$, otherwise let $J$ be the forest obtained from this maximal tree by deleting $e$. Collapse $J$ to get a new unnormalized graph $X'$ with volume $> \frac{1}{3n-3}$. Notice that $X'$ is a minimal graph.  Normalize $X'$ to get $K$. Then $d(X,K) \leq \log \left( \frac{1}{1/(3n-3)} \right) = \log (3n-3)$. The short basis in
$X$ is also short in $K$. 
\end{proof}

\begin{cor}\label{tree_cor}
There are constants $s,c>0$ such that: If $d(\pi(Y), \pi(X))>s$ then
\[d(Y,X) > d(Y, \pi(Y))+d(\pi(Y),\pi(X)) -c\]
\end{cor}
\begin{proof}
By proposition \ref{tree_like_prop1} if $d(\pi(Y), \pi(X))
> s_{\ref{tree_like_prop1}}$ then \[ d(Y, \pi(X))> d(Y, \pi(Y)) +
d(\pi(Y),\pi(X)) - c_{\ref{tree_like_prop1}} \] By proposition
\ref{tree_like_prop2} if $d(\pi(Y),\pi(X))>
s_{\ref{tree_like_prop2}}$ then \[ d(Y,X)> d(Y, \pi(X)) -
c_{\ref{tree_like_prop2}}\] So let $s = \max \{
s_{\ref{tree_like_prop1}}, s_{\ref{tree_like_prop2}} \}$ and $c =
c_{\ref{tree_like_prop1}} + c_{\ref{tree_like_prop2}}$ then $d(Y,X)>
s$ implies \[ \begin{array}{ll}
d(Y,X) > & d(Y, \pi(X)) - c_{\ref{tree_like_prop2}} >
d(Y,\pi(Y)) + d(\pi(Y),\pi(X)) - c_{\ref{tree_like_prop1}} -
c_{\ref{tree_like_prop2}} = \\
& d(Y,\pi(Y)) + d(\pi(Y),\pi(X)) - c
\end{array} \]
\end{proof}

As a corollary we get that the projection is coarsely Lipschitz.

\begin{cor}\label{proj_Lip}
There is a constant $c$ such that for all $X,Y \in \os$: \[d(X,Y) \geq d(\pi(X),\pi(Y)) + c\]
\end{cor}

\subsection{Strongly contracting geodesics}

\begin{defn}
Let $A$ be a subset in $\U$ an asymmetric metric space. The outgoing neighborhood and the incoming neighborhood of $A$ are respectively:
\[ \begin{array}{l}
	N_\delta(A_{\rightarrow}) = \{ x \in \U \mid d(a,x)<\delta \text{ for some } a \in A \} \\
	N_\delta(A_{\leftarrow}) = \{ x \in \U \mid d(x,a)<\delta \text{ for some } a \in A \} 
	\end{array}\]
\end{defn}

\begin{definition}
Let $r>0$. The ball of outward radius $r$ centered at $x$ is \[ \rb{x}{r} = \{ y \in \os \mid d(x,y)<r \}\]
The ball of inner radius $r$ centered at $x$ is \[ \lb{x}{r} = \{ y \in \os \mid d(y,x)<r \}\]
\end{definition}

We will use the ball of outer radius to define the notion of a strongly contracting geodesic in this case. 

\begin{definition}[Strongly contracting geodesics in an asymmetric space]
Let $L$ be a directed geodesic in $\U$, and let 
$\pi_L:\U \to L$ be the closest point projection. $L$ is
$D$-\emph{strongly contracting} if for any ball $\rb{x}{r} \subseteq \U$ disjoint from $L$: $\diam(\pi_L(\rb{x}{r}))<D$.
\end{definition}

\begin{theorem}\label{strongly_cont_cor}
If $f:G \to G$ is a train-track representative of a fully irreducible
outer automorphism $\phi$, then $\LL_f$ is $D$-strongly contracting.
\end{theorem}

\begin{proof}
It is enough to show that there exists a $D>0$ such that $\diam \{ \pi (\rb{Y}{r}) \}
< D$ for $r = d(Y, \pi(Y))$. We'll show that if $X
\in \rb{Y}{r}$ then $d(\pi(Y),\pi(X))<D$ where $D = \max \{
s_{\ref{tree_cor}}, c_{\ref{tree_cor}} \}$. If $X
\in \rb{Y}{r}$ then $d(Y, X)<r$ and by proposition \ref{tree_cor} either
$d(\pi(Y),\pi(X)) \leq s_{\ref{tree_cor}}$ or $d(Y,X)> d(Y,\pi(Y)) +
d(\pi(Y),\pi(X)) - c_{\ref{tree_cor}}$. If the latter occurs then
\[ r > r+d(\pi(Y),\pi(X))-c_{\ref{tree_cor}} \] Thus $d(\pi(Y),\pi(X))< c_{\ref{tree_cor}}$.
\end{proof}

\subsection{The Morse Lemma}

\begin{defn}
The map $\al:[0,l] \to \U$ is a directed $(k,c)$-\emph{quasi-geodesic} if for all $0 \leq t'<t \leq l$ we have 
\[ \frac{1}{k}(t-t') - c \leq d(\al(t), \al(t')) \leq k(t-t') +c\]
\end{defn}

\begin{defn}
A quasi-geodesic $\al:[0,l] \to \U$ is (m,p)-\emph{tame} if for all $0 \leq t' < t \leq l$ we have 
\[ len(\al|_{[t,t']}) \leq m(t-t') + p\]
\end{defn}

\begin{lemma}\label{taming}
For every $(k,c)$-quasi-geodesic $\al:[0,l] \to \U$ there is an $(m,p)$-tame (k',c')-quasi-geodesic $\beta:[0,l] \to \U$ with 
\begin{enumerate}
\item $\beta(0)=\al(0), \beta(l) = \al(l)$
\item $k' = k$, $c' = 2(k+c)$
\item $m = k(k+c)$, and $p = (k+c)(2k^2+2kc +3)$
\item $N_{k+c}(\im \al ) \supseteq \im \beta$ and $N_{k+c}(\im \beta) \supseteq \im \al$
\end{enumerate}
\end{lemma}

The proof of this Lemma \ref{taming} for a symmetric metric space can be found in \cite{BrHa}. The proof for a nonsymmetric space is the same hence we omit it. 

\begin{defn}
A point $x \in \U$ is \emph{high} if there exists a constant $A$ such that $d(x,y) \leq Ad(y,x)$ for all $y \in \U$. A set $S \in \U$ is high if there are constants $A$ so that for all $x\in S$ and $y \in \U$ such that $d(x,y) \leq Ad(y,x)$
\end{defn}

We recall the definition of Hausdorff distance 
\begin{defn}
Let $S,T \subset \U$ be closed. Define the Hausdorff distance 
\[ d_{\text{Haus}}(S,T) = \inf \{ \eps \mid S \subseteq N_\eps(T_\ra) \text{ and }T \subseteq N_\eps(S_\ra) \}\]
\end{defn}

\begin{ML}\label{morse_lemma} If $L$ is a directed, $A$-high, $D$-strongly contracting geodesic in $\U$ and 
  and $\al$ is an $(a,b)$-quasi-geodesic with endpoints on $L$ then
  there exists a constant $C$, depending only on $A,D,a$ and $b$, such that
  $d_{\text{Haus}}(\im L,\im \al)<C$.
\end{ML}

\begin{remark}\label{alm_qg_rmk}
In fact, we only need $\mathcal{\al}$ to satisfy: $len(\al|_{[t',t]})
< a[ d(\al(t'),\al(t)) ]+b$ for the corollary above to hold true. 
\end{remark}

\begin{proof}[Proof of the Morse Lemma]
We may assume that $\al$ is an $(a,b)$-tame quasi-geodesic. 
We fix the following notation. Let $c = \max \{a,b,1\}$, $R = \max \{ d(\al(t), \im L) | t \in [0,l] \}$ and suppose $R>cD$.
Let $[s_1,s_2]$ be a maximum subinterval
such that for every $s \in [s_1, s_2]$: $d(\al(s), \im L) \geq cD$.
Subdivide $[s_1,s_2]$ into: $s_1=r_1, \dots, r_m, r_{m+1}=s_2$ where
$d(\al(r_{i}),\al(r_{i+1}))=2cD$ for $i \leq m$ and
$d(\al(r_{m}),\al(r_{m+1})) \leq 2cD$. Thus:
\begin{equation}\label{qg_eq1}
\text{len}(\al|_{[s_1,s_2]}) \geq \sum_{i=1}^{m+1}
d(\al(r_i),\al(r_{i+1})) \geq 2cDm
\end{equation} 
On the other hand, let $p_i =
\pi_L(\al(r_i))$. Since $d(\al(r_i),p_i) \geq cD$, and since $L$ is $D$-strongly contracting we get
\[ d(p_1,p_{m+1}) \leq D (m+1)\]
So $d(\al(r_1),\al(r_{m+1})) \leq cD + (m+1)D + AcD$ where $L$ is $A$-high. Therefore,
\begin{equation}\label{qg_eq2} 
\text{len}( \al|_{[s_1,s_2]}) \leq c d(\al(s_1),\al(s_2))+c \leq c \bigl( cD + (m+1)D  + A cD\bigr) +c 
\end{equation}
Combining the inequalities \ref{qg_eq1} and \ref{qg_eq2} we get:
\[ 2mcD \leq   c^2 D + (m+1)cD  + A c^2 D + c \]
After some manipulation we get: $m \leq c +c \cdot A+ \frac{1}{D}+1 = K$. \\

Hence $len(\al|_{[s_1,s_2]}) \leq m 2cD < 2KcD$. Thus for each $s \in [s_1,s_2]$: 
\[ 
d(\al(s),\im L)  <  d(\al(s),\al(s_2)) + d(\al(s_2),\im L) \leq 
len(\al|_{[s_1,s_2]})  + cD <  2KcD + cD 
\]
Since $\im L$ is $A$-high we get that there is a constant $C$ so that $d_{\text{Haus}}(\im L, \im \al) < C$
\end{proof}

Since $\LL_f$ is periodic there is an $\eps$ so that $\im \LL_f \subseteq \os^{\geq \eps}$. By Theorem \ref{almost_sym} the set $\LL_f$ is $A$-high. Thus applying the Morse Lemma we get

\begin{theorem}\label{morse_axis}
$\LL_f$ is a Morse geodesic: 
For any $(a,b)$-quasi-geodesic $\Q$ with endpoints on $\LL_f$ there is a $C$ that depends only on $a,b,\eps$ and $D_{\ref{strongly_cont_cor}}$ so that $d_{\text{Haus}}(\im \LL_f, \im \Q) < C$.
\end{theorem}


\section{Applications}

\subsection{Morse Geodesics in the Cayley Graph of $\out$}\label{MorseSection}

Let $\cayout$ be the Cayley graph of $\out$ with the generating set of Whitehead automorphisms  $\{ \phi_i \}_{i=1}^{N}$, i.e. it's vertices $V(\cayout)$ are the elements of  $\out$ and $\psi_1,\psi_2 \in \out$ are connected by an edge if there is a Whitehead generator $\phi_i$ so that $\psi_1 = \phi_i \circ \psi_2$ (we want $\out$ to act on the right). Let $\phi$ be a fully
irreducible outer automorphism. Let $f:G \to G$ be a stable 
train-track representative for $\phi$. Choose an embedding $\iota: \cayout \into
\os$ as follows. Let $L$ be the axis for $\phi$ in $\cayout$. Choose some
vertex $\psi \in L$ and define $\iota(\psi)=G$. Extend $\iota$ to the
vertices of $\cayout$ equivariantly and to the edges of $\cayout$ by mapping them onto some
geodesic between the images of their endpoints. \\

Let $M = \max \{
d_{\os}(\iota(\text{id}), \iota(\phi_i)) \mid \phi_i \text{ is a
generator}\}$ then for the vertices of $\cayout$ we have: 
\[ d_{\os}(\iota(\psi_1), \iota(\psi_2)) \leq M \cdot d_{\cayout}(\psi_1,
   \psi_2) \]
For other points in $\cayout$ a similar inequality holds (by adding $2M$). The reverse inequality is false.

\begin{example}
  $\psi_1 = \textup{id}$ and $\psi_2 =
  \left\{ \begin{array}{l}
      x \to x \\
      y \to xy^m
    \end{array} \right.$
Suppose $\iota(\psi_1) = R$ is a bouquet of $2$ circles each of length $\frac{1}{2}$ with the identity marking. Then  $d_{\os}(\iota(\psi_1), \iota(\psi_2)) = d_{\os}(R, R \cdot \psi_2) = \log( \frac{(m+1)/2}{1/2}) = \log(m+1)$, while $d_{\cayout}(\psi_1,\psi_2) = m$.
\end{example}

However, for points on $L$ (the axis for $\phi$) distances in $\os$
coarsely correspond to distances in $\cayout$. 

 Even though distances in $\cayout$ are larger than their images in $\os$ they cannot be arbitrarily larger, as the next lemma shows. 

\begin{lemma}\label{bdd_dist_in_cayley}
For every $a>0$ there is a $b>0$ such that: If $d_{\os}(\iota(\psi),
\iota(\chi))<a$ then $d_{\cayout}(\psi,\chi)<b$. 
\end{lemma}

\begin{proof}
By equivariance, $\iota(V(\cayout))$ is discrete. Thus the set $\mathcal{A} = \{ \psi \mid d_{\os}( \iota(\id) , \iota(\psi) ) < a \}$ is finite. Let $b = \max \{
d_{\cayout}( \id, \psi ) \mid \psi \in \mathcal{A} \}$. Suppose 
$d_{\os}( \iota(\psi), \iota(\chi) )<a$ then $d_{\os}( \iota(\id),
\iota(\chi \psi^{-1}) )<a$, so $d_{\cayout}(\id, \chi \psi^{-1})<b$ and
$d_{\cayout}(\psi, \chi)<b$
\end{proof}

\begin{theorem}\label{Morse geod in Cayley}
If $L$ is the axis in $\cayout$ of a fully irreducible automorphism then $L$  is a Morse geodesic in $\cayout$.
\end{theorem}
\begin{proof}
Let $\al$ be an $(a,b)$-quasi geodesic in $\cayout$ with endpoints on $L$. We may assume that $\al$ is tame. Consider $\Q = \iota \circ \al$, then $\text{len}_{\os}\Q|_{[t,t']} \leq M \cdot \text{len}_{\cayout}\al|_{[t,t']} \leq  
M a (t-t')+M b$. By remark \ref{alm_qg_rmk}, there exists a $d$, depending only on $a,b,M,D$ and $\eps$ (where $\LL_f$ is in the $\eps$ thick part of $\os$) such that $d_{\text{Haus}}(\im \Q, \im \LL_f) < d$. By lemma
\ref{bdd_dist_in_cayley} we have $d_{Haus}(\im L, \im \al)< D$ for some $D$ depending only on $d$. 
\end{proof}

\subsection{The Asymptotic Cone of $\cayout$}

\begin{defn}
A nonprincipal maximal ultrafilter $\omega$ on the integers is a nonempty
collection of subsets of $\ZZ$ so that:
\begin{itemize}
	\item $\omega$ is closed under inclusion 
	\item $\omega$ is closed under finite intersection
	\item $\omega$ does not contain any finite sets 
	\item $A \subset \ZZ$, if $A \notin \omega$ then $\ZZ \setminus
A \in \omega$
\end{itemize}
\end{defn}

\begin{defn}
Let $\omega$ be a nonprinciple maximal ultrafilter on the integers. Let $(X_i, x_i, d_i)$ be a sequence of based metric spaces. Define the following pseudo-distance on $\prod_{i \in \NN} X_i$:
\[ d_\omega( \{ a_i \},  \{b_i \}) =  \displaystyle \lim_\omega d_{X_i}(a_i, b_i)\]
The ultralimit of $(X_i,x_i)$ is then 
\[ \displaystyle \lim_\omega (X_i, x_i, d_i) =  \{ y \in \prod_{i \in \NN}X_i \mid d_\omega (y, \{x_i\}) < \infty \}/ \sim \]
where $y \sim y'$ if $d_\omega(y,y') = 0$. 
\end{defn}

Consider a space $X$, a point $x \in X$ and a sequence of  integers $k_i$ such that $\displaystyle \lim_{i \to \infty} k_i = \infty$. 
\begin{defn}
The asymptotic cone of $(X, x, \{ k_i \})$ relative to the ultrafilter $\omega$ is:
\[ \Cone_\omega(X, x, k_i) = \displaystyle \lim_\omega \left( X,x,\frac{1}{k_i}d_X(\cdot, \cdot ) \right) \]
\end{defn}

The asymptotic cone of a geodesic metric space is a geodesic metric space. 

We recall the following definition made in \cite{DS}.

\begin{definition}
Let $W$ be a complete metric space and let $\mathcal{P}$ be a collection of  closed geodesic subsets (called pieces). The space $W$ is said to be tree-graded with respect to $\mathcal{P}$ if the following properties are satisfied:
\begin{enumerate}
\item The intersection of two pieces is either empty of a single point.
\item Every simple geodesic triangle in $X$ is contained in one piece.
\end{enumerate}
The arcs starting in a given point $w \in W$ intersecting each piece in at most one point compose a real tree called \emph{a transversal tree}.
\end{definition}

In particular, if $p$ is in a transversal tree then $p$ is a cut point of $W$. 

\begin{theorem}[Proposition 3.24 in \cite{DMS}]
Let $X$ be a metric space and let $\mathfrak{q}$ be a quasi-geodesic. The following are equivalent:
\begin{itemize}
\item The image of $\mathfrak{q}$ in every asymptotic cone of $X$ is either empty or contained in a transversal tree of $X$ for some tree graded structure. 
\item $\mathfrak{q}$ is a Morse quasi-geodesic.
\end{itemize}
\end{theorem}

We get the following corollary from Theorem \ref{Morse geod in Cayley}.

\begin{cor}
The image of an axis of an irreducible automorphism in $\Cone_\omega \cayout$ is either empty or is contained in a transversal tree for some tree graded structure of $\Cone_\omega \cayout$.
\end{cor}

\begin{remark}
It is tempting to try to define the asymptotic cone of Outer Space itself. One would like to conclude that the cone is itself an asymmetric metric space. We choose a basepoint $x_0 \in \os$ and define:
\[ \displaystyle \lim_\omega (\os,\{ x_0 \}, d_\text{Lip}) =  \left\{ \left. \{ y_i \} \in \prod_{i \in \NN}\os ~ \right| ~ d_\omega (\{ y_i \}, \{x_0\}) < \infty \right\} \]
Note that by the asymmetry theorem $d_\omega (\{x_0\}, \{ y_i \})< \infty$ but now it does not make sense to mod out by the equivalence relation $ y \sim z$ if $d(y,z) = 0$ because $d(z,y)$ might be positive. For example, if we choose $y_i$ so that $d_\text{Lip}(y_i, x_0) = i$ and $d_\text{Lip}(x_0, y_i)$ is bounded then $d(\{ y_i\}, \{ x_0\}) = 1$ and $d(\{ x_0\}, \{y_i\}) = 0$. We could always restrict our attention to the thick part of Outer Space where the distances are almost symmetric. However, this space is quasi-isometric to $\cayout$ so we will not get anything new. 
\end{remark}

\subsection{Divergence in Outer Space}

\begin{definition} 
Let $\gamma_1, \gamma_2$ be two geodesic rays in $\os$, with $\gamma_1(0) = \gamma_2(0)=x$. The \emph{divergence} function from $\gamma_1$ to $\gamma_2$ is:
\[ \text{div}(\gamma_1,\gamma_2,t) = 
\inf  \left\{ \text{length}(\gamma) \left| \begin{array}{c} \gamma: [0,1] \to \os \smallsetminus \lb{x}{r} \\
\gamma(0) = \gamma_1(t), \gamma(1) = \gamma_2(t) \end{array} \right\}  \right. \] 
If $f(t)$ is a function such that:
\begin{enumerate} 
	\item\label{itemA} for every $\gamma_1, \gamma_2$: $\text{div}(\gamma_1,\gamma_2,t) \prec f(t)$ (we use $g(t) \prec f(t)$ to denote the relationship $f(t) \leq c \cdot g(t)+ c'$ for all $t$) 
	\item\label{itemB} there exist geodesics $\gamma_1, \gamma_2$ such that  $\text{div}(\gamma_1,\gamma_2,t) \asymp f(t)$.
	\end{enumerate}
then we say that the divergence function of $\os$ is on the order of $f(t)$. If only \ref{itemA} holds we say that $f$ is an upper bound for the divergence of $\os$,   if only \ref{itemB} holds we say that $f$ is a lower bound for the divergence of $\os$. 
\end{definition}

Behrstock \cite{Beh} proved that the divergence in $\mcg(S)$ is quadratic. Duchin and Rafi \cite{DR} prove that the divergence in \teich space is quadratic. The proof that the divergence is at least quadratic in the Outer Space setting needs very little modification, but we include it for the reader's convenience.

\begin{prop}\label{quad_prop}
Let $\gamma$ be a path in $\os$, from $x$ to $y$ with $d(\pi(x),\pi(y) )= 2R$. Let $z$ the point on $\LL_f$ in the middle of the segment $[\pi(x), \pi(y)] \inn \LL$. Further assume that the image of $\gamma$ lies outside the ball $\lb{z}{R}$. If $R> 2D_{\ref{strongly_cont_cor}}$ then there is a constant $c$ such that $Len(\gamma) \geq c R^2$ where $c$ only depends on the constants $D_{\ref{strongly_cont_cor}}$ and $c_{\ref{proj_Lip}}$.
\end{prop}
\begin{proof}
Subdivide $\gamma$ into $n>1$ subsegments $I_1, I_2, \dots , I_n$, each of which has length $\frac{R}{2}$ except for  possibly $Len(I_n) \leq \frac{R}{2}$. Therefore $Len(\gamma) \geq (n-1) \frac{R}{2}$. Let $\LL_0$ be the subsegment of $\LL_f$ centered at $z$ of length $R$. Since $\LL$ is $b$-contracting for $b = D_{\ref{strongly_cont_cor}}$ then $\LL_0$ is $b'$-contracting for $b' = b+ 4c_{\ref{proj_Lip}} +3$ (see Lemma 3.2 in \cite{BFu}). Each segment $I_j$ is contained in a ball $\lb{x' }{R/2}$ disjoint from $\LL_0$. Thus the length of each $\pi(I_j) \leq b'$, since these segments cover $\LL_0$ we get $R \leq nb'$. Therefore $Len(\gamma) \geq (n-1)\frac{R}{2} > \left( \frac{R}{b'} -1 \right) \frac{R}{2}=  \frac{1}{2b'}R^2 - \frac{1}{2}R$.
\end{proof}

The exact behavior of the divergence function of $\os$ remains open. 

\subsection{The Behrstock Inequality} 

In this section let  $\phi,\psi \in \out$ be two irreducible outer automorphisms  and $f,g$ their respective
train-track representatives. Denote $A = \LL_f$, $B = \LL_g$ and
$p_A = \pi_f$ and $p_B = \pi_g$.  Our first goal is to show that either
$A,B$ are parallel or the diameter of $p_A(B)$ is bounded, and we would
like to understand what the bound depends on. 
We introduce the following notation for the next lemma: if $x,y \in A$
denote by $[x,y]_A$ the subinterval of $A$ whose endpoints are $x$ and $y$.

\begin{lemma}\label{far_proj}
There exist constants $c,d$ such that if $x,y \in B$
with $d(p_A(x), p_A(y))>c$, then \[ [p_A(x),p_A(y)]_A \subset
\mathcal{N}_d(B)\] $c,d$ depend only on the constants
$s_{\ref{tree_like_prop2}},
c_{\ref{tree_like_prop2}}$ applied to $A$ and $B$ and on
$D_{\ref{almost_sym}}(\theta)$ where $\theta$ 
is small enough so that $A,B \subset \os(\theta)$.
\end{lemma}

\begin{proof}
Let $c_1 =  D_{\ref{almost_sym}}(\theta)$ from Theorem \ref{almost_sym}, thus for
  all $z,w \in \os(\theta)$: $d(w,z)< c_1 \cdot
  d(z,w)$. Let $s_A, c_A$ be the constants from
  lemma \ref{tree_like_prop2} applied to $A$, thus if $z,w$ are points such that
  $d(p_A(z),p_A(w))>s_A$ then $d(z,w)> d(z,p_A(z))+ d(p_A(z),p_A(w)) -
  c_A$. Let $a = 1+(c_1)^2$, $b=c_A(1+c_1)$ and $d =
  c_{\ref{morse_axis}}(a,b)$ from Theorem \ref{morse_axis} applied to $B$,
  i.e., for every $(a,b)$-quasi-geodesic
  $\mathcal{Q}$ with endpoints on $B$,
  $N_d(B) \supset \im \mathcal{Q}$. We prove that $[x,p_A(x)]\cup
  [p_A(x),p_A(y)]_A \cup [p_A(y),y]$ is an $(a,b)$-quasi-geodesic. \\

First note 
\begin{equation} \label{eq_B_1}
d(x,y) > d(x, p_A(x)) + d(p_A(x),p_A(y)) - c_A
\end{equation}
Similarly,
\[ \begin{array}{l}
  d(y,x) > d(y,p_A(y)) + d(p_A(y),p_A(x)) - c_A \\
  d(y,x) > d(y,p_A(y)) - c_A > \frac{1}{c_1}d(p_A(y),y) -c_A 
\end{array} \] So
\begin{equation}
  \label{eq_B_2} (c_1)^2 \cdot  d(x,y)> c_1 d(y,x) > d(p_A(y),y) - c_1c_A
\end{equation}
Adding equations \ref{eq_B_1} and \ref{eq_B_2} we get
\[ \begin{array}{l}
 (1+(c_1)^2) d(x,y) >  \\
 d(x,p_A(x))+ d(p_A(x),p_A(y)) + d(p_A(y),y) -c_A(1+c_1)
\end{array} \]
Therefore $[x,p_A(x)] \cup [p_A(x),p_A(y)]_A \cup [p_A(y),y]$ is a
$(1+(c_1)^2, c_A(1+c_1))$-quasi-geodesic. Hence $ [p_A(x),p_A(y)]_A \subset N_d(B)$.
\end{proof}

The next Lemma is motivated by the following
observation. Let $X$ is a proper metric space with a properly
discontinuous isometric $G$-action. Let $g,h \in G$ be hyperbolic isometries of
$X$ and let $A_g , A_h$ denote their axes. Then for every $d$ there is a constant $k$ which depends only on $d, tr(g), tr(h)$ such that either $A_g, A_h$ are parallel, or the length of $A_g \cap N_d(A_h)$  is shorter than $k$. In our case, Outer Space is not proper. The closure of a ball $\rb{x}{r} = \{ y \in \os \mid d(x,y) < r \}$ need not be compact. However 

\begin{claim}
The closure of the ball $\lb{x}{r} = \{ y \in \os \mid d(y,x) < r \}$ is  compact. 
\end{claim}
\begin{proof}
For each $y \in \overline{\lb{x}{r}}$ and for all conjugacy classes $\al$, $l(\al,y) \geq \frac{l(\al,x)}{e^r}$. Thus if $\theta$ is the length of the shortest loop in $x$ then $l(\al,y)\geq \frac{\theta}{e^r}$. So $\partial{\os} \cap \overline{\lb{x}{r}} = \emptyset$ and since $\overline{\os}$ is compact then the closure of $\lb{x}{r}$ in Outer Space is compact. 
\end{proof}

Recall that the $\out$ action is properly discontinuous. Thus for every $r$ there is a number $N_r$ such that $\lb{x}{r}$ contains no more than $N_r$ points of any orbit.

\begin{definition}
Let $A,B$ be two axes in $\os(\theta)$ and $d>0$. We will define closed connected subsets $A_B \subseteq A$ and $B_A \subseteq B$ and $R>0$ (depending on $d$) with 
\[ \begin{array}{llllll}
  \forall x \in A_B & & d(x,p_B(x))<R & \mbox{ and }& d(p_B(x),x)<R \\
  \forall x \in B_A & & d(x,p_A(x))<R & \mbox{ and }& d(p_A(x),x)<R
\end{array} \]
as follows. Let
\[ A'_B(d) = \{ x \in A \mid d(x,p_B(x)) \leq d \mbox{ and } d(p_B(x),x) \leq d \} \]
Let $A_B(d)$ be the smallest connected closed set in $A$ containing
$A'_B(d)$. We claim that for any $a \in A_B(d)$: $d(a,p_B(a))< r$ for
some $r$ that depends on $d$ and on the constants from Lemma
\ref{morse_axis}. The reason is that if $a \in [b,c]_A$ for some $b,c
\in A'_B(d)$ then $[p_B(b), b]_B \cup [b,c]_A \cup [c,p_B(c)]$ is a
$(1,2d)$ quasi-geodesic so it is contained in the $r$ neighborhood of
$B$. Furthermore, $d(p_B(a),a)< c_1 r$ where $c_1 = D_{\ref{almost_sym}}(\theta)$. 

Let $B'_A = \{ b \mid b = p_B(x) \mbox{ for } x \in A_B \}$ then for $b \in B'_A$ there is an $a \in A_B$ with $b=p_B(a)$ hence $d(b,
p_A(b)) \leq d(b,a) = d(p_B(a),a) < c_1 r$ and $d(p_A(b),b)< (c_1)^2 r$. Let $B_A$ is the
smallest closed connected set containing $B'_A$ then $\forall b \in
B_A$ we have $d(b,p_A(b))< R$ for $R$ obtained from Lemma
\ref{morse_axis} applied to $A$ and $(c_1)^2r$. This completes our definition.
\end{definition}

\begin{lemma}\label{nbhd_not_contain_ray}
For every $d$, there exists a constant $c$ such that either $f,g$ have
common powers, or the length of $A_B(d)$ is smaller than $c \max\{ tr(f), tr(g) \}$.
\end{lemma}

\begin{proof}
Let $k$ denote the length of $A_B(d)$. Let $a$ be the
leftmost point on $A_B(d)$ assuming that $f$ 
translates points to the right. Without loss of generality, assume $f$
and $g$ translate points in the same direction. Denote $b=p_B(a)$.
 For each $i \leq \frac{k}{tr(f)}$ there is a unique $j$
such that 
\[d(b, p_B(a f^i) g^{-j} )<tr(g)\] Since $d(p_B(a f^i), a
f^i)<R$ then $d(p_B(a f^i) g^{-j}, a f^i g^{-j} )<R$ hence 
\[ \begin{array}{lll} 
d(a, a f^i g^{-j}) & \leq & d(a,b) + d(b, p_B(a f^i) g^{-j} ) + d(p_B(a f^i) g^{-j}, a
f^i g^{-j}) \\ & \leq & R + tr(g) + R =  tr(g) + 2R
\end{array}\] 
 Therefore, $d(a f^ig^{-j}, a)< c_1(tr(g) +2R)<
c_1(M+2R)$ where  $c_1=D_{\ref{almost_sym}}(\theta)$. \\
Let $r = c_1(M+2R)$ then there are no more than $N_r$ translates of
$a$ in $\lb{a}{r}$, but for each $i< \frac{k}{tr(f)}$: $d(a f^ig^{-j},
a)< r$. Therefore, either $k < tr(f) N_r < M N_r$ or there exists $i,j,
m,l$ such that $f^{i} g^{-j} = f^{m}g^{-l}$ hence $f,g$ have common powers. 
\end{proof}

\begin{cor}\label{proj_of_geod}
There exists a constant $k$, depending only on the constants from 
Lemma \ref{nbhd_not_contain_ray} and Lemma \ref{far_proj}, such that either
$f,g$ have common powers or \[ \diam \{ p_A(B) \} <k\]
\end{cor}

\begin{proof}
Let $\{x_i\},\{y_i\}$ be sequences on $B$ so that $x_i$ converges to
one end of $B$ and $y_i$ to the other. If 
$d(p_A(x_i), p_A(y_i))> c_{\ref{far_proj}}$ then $[p_A(x_i),p_A(y_i)]
\subseteq N_d(B)$. Thus $A' = [p_A(x_i),p_A(y_i)]_A$ is
contained in $A_B(c_1d)$. Therefore by Lemma
\ref{nbhd_not_contain_ray} there is a $c = c_{\ref{nbhd_not_contain_ray}} \max\{tr(f), tr(g)\}$
such that either $f,g$ have common powers or the length of
$A_B(c_1d)<c$ and hence the length of $A'$ is smaller than $c$. 
\end{proof}

Let us go back for a moment to the surface case. We denote by
$\mathcal{M}(S)$ the marking complex of $S$. Let $Y,Z$ be
subsurfaces of $S$, denote by $\mathcal{C}(Z), \C(Y)$ the curve
complexes of $Z,Y$. For definitions of the curve complex and the marking complex consult \cite{Beh}. Define the projections (slightly abusing notation) 
$p_Y:\mathcal{C}(S) \to \C(Y)$, $p_Y: \mathcal{M}(S) \to \C(Y)$ and $p_Z:\mathcal{C}(S) \to \C(Z)$, $p_Z: \mathcal{M}(S) \to \C(Z)$. In Theorem 4.3 of  \cite{Beh}, Behrstock proved that if $Y,Z$ are overlapping subsurfaces of $S$, neither of which is an annulus, then for any marking $\mu$ of $S$: 
\[ d_{\mathcal{C}(Y)}(p_Y(\partial Z), p_Y(\mu) ) > M \implies
d_{\mathcal{C}(Z)}(p_Z(\partial Y),p_Z(\mu) )<M \] And the constant $M$
depends only on the topological type of $S$.  In other words, if one
projection is large then the other must be small. We prove an
analogous estimate for our projections.\\ 

Suppose $f,g,h$ are stable train-track maps representing fully irreducible
automorphisms and $A,B,C$ are their axes. Suppose that no two of
these automorphisms have common powers. We define the coarse distance
from $B$ to $C$ with respect 
to $A$ as \[ d_A(B,C) = \diam\{ p_A(C) \cup p_A(B)\}\]

\begin{lemma}
There exists a constant $M>0$ depending only on the constants from Lemma \ref{tree_like_prop2} and Corollary \ref{proj_of_geod} such that at most one of $d_A(B,C), d_B(A,C)$ and $d_C(A,B)$ is greater than $M$.
\end{lemma}

\begin{proof}
Let $s_A, c_A,s_B, c_B,s_C, c_C$ be the constants from lemma
\ref{tree_like_prop2} applied to any of the geodesics $A,B,C$
respectively. Let $b \geq k_{\ref{proj_of_geod}}$ the constant from
Corollary \ref{proj_of_geod} applied to any two of the three
geodesics. Let $M > \max \{s_A, c_A,s_B, c_B,s_C, c_C \} +2b$. We claim that if $d_B(A,C)> M$ then 
$d_C(A,B)<M$. Assume by way of contradiction that both are greater
than $M$. Let $y \in A$ and $q \in B$ such that $d(y,q) =
d_{\text{Haus}}(A,B)$. Let $z = p_C(y) \sim p_C(A)$, $p = p_B(z) \sim
p_B(C)$ and $x = p_C(q) \sim p_C(B)$ (see Figure \ref{BehInFig}).

\begin{figure}[ht]
\begin{center}
\input{BehIn.pstex_t}
\caption{\label{BehInFig}If $d(p,q)>M$ then $d(x,z)<M$.}
\end{center}
\end{figure}
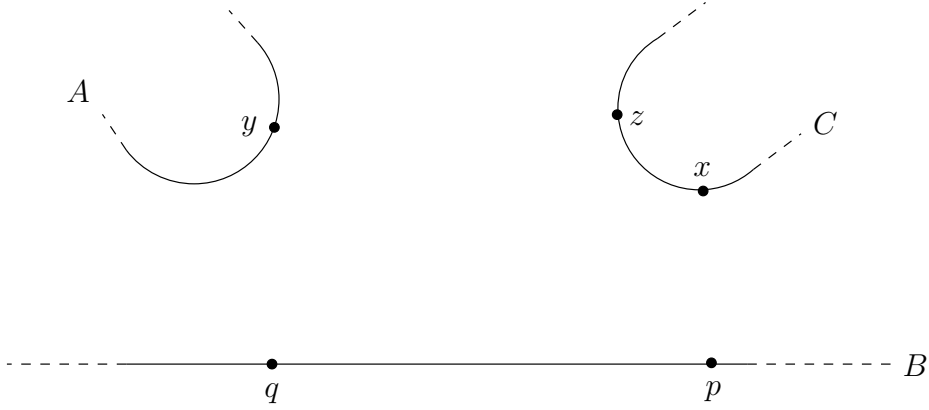

Because $d(p_C(y),p_C(q)) = d(z,x)> M - 2b  > s_C$:
\[ d(y,q) > d(y,z) + d(z,x) - c_C > d(y,z) + M -2b - c_C \]
Since $d(p_B(y),p_B(z)) = d(q,p) > M - 2b > s_B $ we have
\[ d(y,z) > d(y,q) + d(q,p) - c_B > d(y,q) + M -2b - c_B \]
Therefore
\[ d(y,q) > d(y,q) + 2M - c_C - c_B -4b \] which
implies $2M < c_C + c_B +4b$ which is a contradiction.
\end{proof}

\begin{theorem} 
Let $\phi_1, \dots , \phi_k$ be fully irreducible outer automorphisms such that no two have common powers and $f_1, \dots , f_k$ their respective train track representatives with axes $A_1, \dots , A_k$. Let $\mathcal{F}$ be the set of  translates of $A_1, \dots , A_k$ under the action of $\out$. Then there exists a constant $M>0$ such that for any $B, C, D \in \mathcal{F}$ then
\[ d_A(B,C) > M \implies d_B(A,C)< M \]
\end{theorem}

\bibliographystyle{siam}
\bibliography{ref}

\end{document}

%% file: optimal_map.pstex_t
\begin{picture}(0,0)%
\includegraphics{optimal_map.pstex}%
\end{picture}%
\setlength{\unitlength}{3947sp}%
\begingroup\makeatletter\ifx\SetFigFont\undefined%
\gdef\SetFigFont#1#2#3#4#5{%
  \reset@font\fontsize{#1}{#2pt}%
  \fontfamily{#3}\fontseries{#4}\fontshape{#5}%
  \selectfont}%
\fi\endgroup%
\begin{picture}(5064,1804)(611,-1185)
\put(3248,-218){\makebox(0,0)[lb]{\smash{{\SetFigFont{12}{14.4}{\rmdefault}{\mddefault}{\updefault}{\color[rgb]{0,0,0}$f$}%
}}}}
\put(2214,112){\makebox(0,0)[lb]{\smash{{\SetFigFont{12}{14.4}{\rmdefault}{\mddefault}{\updefault}{\color[rgb]{0,0,0}$e_2$}%
}}}}
\put(1598,-609){\makebox(0,0)[lb]{\smash{{\SetFigFont{12}{14.4}{\rmdefault}{\mddefault}{\updefault}{\color[rgb]{0,0,0}$v$}%
}}}}
\put(931,119){\makebox(0,0)[lb]{\smash{{\SetFigFont{12}{14.4}{\rmdefault}{\mddefault}{\updefault}{\color[rgb]{0,0,0}$e_1$}%
}}}}
\put(4043,231){\makebox(0,0)[lb]{\smash{{\SetFigFont{12}{14.4}{\rmdefault}{\mddefault}{\updefault}{\color[rgb]{0,0,0}$w'$}%
}}}}
\put(4201,-151){\makebox(0,0)[lb]{\smash{{\SetFigFont{12}{14.4}{\rmdefault}{\mddefault}{\updefault}{\color[rgb]{0,0,0}$a_1$}%
}}}}
\put(4741,472){\makebox(0,0)[lb]{\smash{{\SetFigFont{12}{14.4}{\rmdefault}{\mddefault}{\updefault}{\color[rgb]{0,0,0}$w$}%
}}}}
\put(5558,-144){\makebox(0,0)[lb]{\smash{{\SetFigFont{12}{14.4}{\rmdefault}{\mddefault}{\updefault}{\color[rgb]{0,0,0}$a_3$}%
}}}}
\put(4854,-151){\makebox(0,0)[lb]{\smash{{\SetFigFont{12}{14.4}{\rmdefault}{\mddefault}{\updefault}{\color[rgb]{0,0,0}$a_2$}%
}}}}
\end{picture}%

%% file: non_symmetric.pstex_t
\begin{picture}(0,0)%
\includegraphics{non_symmetric.pstex}%
\end{picture}%
\setlength{\unitlength}{3947sp}%
\begingroup\makeatletter\ifx\SetFigFont\undefined%
\gdef\SetFigFont#1#2#3#4#5{%
  \reset@font\fontsize{#1}{#2pt}%
  \fontfamily{#3}\fontseries{#4}\fontshape{#5}%
  \selectfont}%
\fi\endgroup%
\begin{picture}(5302,1426)(867,-1416)
\put(4349,-1360){\makebox(0,0)[lb]{\smash{{\SetFigFont{9}{10.8}{\rmdefault}{\mddefault}{\updefault}{\color[rgb]{0,0,0}$\frac{m-1}{m}$}%
}}}}
\put(5756,-1125){\makebox(0,0)[lb]{\smash{{\SetFigFont{9}{10.8}{\rmdefault}{\mddefault}{\updefault}{\color[rgb]{0,0,0}$\frac{1}{m}$}%
}}}}
\put(1200,-1199){\makebox(0,0)[lb]{\smash{{\SetFigFont{9}{10.8}{\rmdefault}{\mddefault}{\updefault}{\color[rgb]{0,0,0}$\frac{1}{2}$}%
}}}}
\put(2468,-1209){\makebox(0,0)[lb]{\smash{{\SetFigFont{9}{10.8}{\rmdefault}{\mddefault}{\updefault}{\color[rgb]{0,0,0}$\frac{1}{2}$}%
}}}}
\put(1264,-493){\makebox(0,0)[lb]{\smash{{\SetFigFont{9}{10.8}{\rmdefault}{\mddefault}{\updefault}{\color[rgb]{0,0,0}$y$}%
}}}}
\put(2474,-505){\makebox(0,0)[lb]{\smash{{\SetFigFont{9}{10.8}{\rmdefault}{\mddefault}{\updefault}{\color[rgb]{0,0,0}$x$}%
}}}}
\put(1836,-151){\makebox(0,0)[lb]{\smash{{\SetFigFont{9}{10.8}{\rmdefault}{\mddefault}{\updefault}{\color[rgb]{0,0,0}$X$}%
}}}}
\put(4493,-376){\makebox(0,0)[lb]{\smash{{\SetFigFont{9}{10.8}{\rmdefault}{\mddefault}{\updefault}{\color[rgb]{0,0,0}$y$}%
}}}}
\put(5756,-569){\makebox(0,0)[lb]{\smash{{\SetFigFont{9}{10.8}{\rmdefault}{\mddefault}{\updefault}{\color[rgb]{0,0,0}$x$}%
}}}}
\put(5126,-99){\makebox(0,0)[lb]{\smash{{\SetFigFont{9}{10.8}{\rmdefault}{\mddefault}{\updefault}{\color[rgb]{0,0,0}$Y$}%
}}}}
\end{picture}%

%% file: long_pieces.pstex_t
\begin{picture}(0,0)%
\includegraphics{long_pieces.pstex}%
\end{picture}%
\setlength{\unitlength}{3947sp}%
\begingroup\makeatletter\ifx\SetFigFont\undefined%
\gdef\SetFigFont#1#2#3#4#5{%
  \reset@font\fontsize{#1}{#2pt}%
  \fontfamily{#3}\fontseries{#4}\fontshape{#5}%
  \selectfont}%
\fi\endgroup%
\begin{picture}(6236,1330)(2089,-1008)
\put(2168,-281){\makebox(0,0)[lb]{\smash{{\SetFigFont{10}{12.0}{\rmdefault}{\mddefault}{\updefault}{\color[rgb]{0,0,0}$\beta_1$}%
}}}}
\put(2973,-683){\makebox(0,0)[lb]{\smash{{\SetFigFont{10}{12.0}{\rmdefault}{\mddefault}{\updefault}{\color[rgb]{0,0,0}$\beta_2$}%
}}}}
\put(6122, 54){\makebox(0,0)[lb]{\smash{{\SetFigFont{10}{12.0}{\rmdefault}{\mddefault}{\updefault}{\color[rgb]{0,0,0}$h_0$}%
}}}}
\put(3676,-194){\makebox(0,0)[lb]{\smash{{\SetFigFont{10}{12.0}{\rmdefault}{\mddefault}{\updefault}{\color[rgb]{0,0,0}$\beta$}%
}}}}
\put(5030,-945){\makebox(0,0)[lb]{\smash{{\SetFigFont{10}{12.0}{\rmdefault}{\mddefault}{\updefault}{\color[rgb]{0,0,0}$G_0$}%
}}}}
\put(7549,-877){\makebox(0,0)[lb]{\smash{{\SetFigFont{10}{12.0}{\rmdefault}{\mddefault}{\updefault}{\color[rgb]{0,0,0}$F$}%
}}}}
\end{picture}%

%% file: tree_prop1.pstex_t
\begin{picture}(0,0)%
\includegraphics{tree_prop1.pstex}%
\end{picture}%
\setlength{\unitlength}{3947sp}%
\begingroup\makeatletter\ifx\SetFigFont\undefined%
\gdef\SetFigFont#1#2#3#4#5{%
  \reset@font\fontsize{#1}{#2pt}%
  \fontfamily{#3}\fontseries{#4}\fontshape{#5}%
  \selectfont}%
\fi\endgroup%
\begin{picture}(5877,3182)(1639,-3531)
\put(3174,-2094){\makebox(0,0)[lb]{\smash{{\SetFigFont{12}{14.4}{\rmdefault}{\mddefault}{\updefault}{\color[rgb]{0,0,0}$\pi(Y)$}%
}}}}
\put(6234,-2101){\makebox(0,0)[lb]{\smash{{\SetFigFont{12}{14.4}{\rmdefault}{\mddefault}{\updefault}{\color[rgb]{0,0,0}$X$}%
}}}}
\put(2776,-825){\makebox(0,0)[lb]{\smash{{\SetFigFont{12}{14.4}{\rmdefault}{\mddefault}{\updefault}{\color[rgb]{0,0,0}$Y$}%
}}}}
\put(7501,-1936){\makebox(0,0)[lb]{\smash{{\SetFigFont{12}{14.4}{\rmdefault}{\mddefault}{\updefault}{\color[rgb]{0,0,0}$\mathcal{L}$}%
}}}}
\end{picture}%

%% file: tree_prop2.pstex_t
\begin{picture}(0,0)%
\includegraphics{tree_prop2.pstex}%
\end{picture}%
\setlength{\unitlength}{3947sp}%
\begingroup\makeatletter\ifx\SetFigFont\undefined%
\gdef\SetFigFont#1#2#3#4#5{%
  \reset@font\fontsize{#1}{#2pt}%
  \fontfamily{#3}\fontseries{#4}\fontshape{#5}%
  \selectfont}%
\fi\endgroup%
\begin{picture}(5877,3174)(1714,-3523)
\put(6518,-909){\makebox(0,0)[lb]{\smash{{\SetFigFont{12}{14.4}{\rmdefault}{\mddefault}{\updefault}{\color[rgb]{0,0,0}$X$}%
}}}}
\put(6158,-2079){\makebox(0,0)[lb]{\smash{{\SetFigFont{12}{14.4}{\rmdefault}{\mddefault}{\updefault}{\color[rgb]{0,0,0}$\pi(X)$}%
}}}}
\put(2723,-788){\makebox(0,0)[lb]{\smash{{\SetFigFont{12}{14.4}{\rmdefault}{\mddefault}{\updefault}{\color[rgb]{0,0,0}$Y$}%
}}}}
\put(7576,-1914){\makebox(0,0)[lb]{\smash{{\SetFigFont{12}{14.4}{\rmdefault}{\mddefault}{\updefault}{\color[rgb]{0,0,0}$\mathcal{L}$}%
}}}}
\end{picture}%

%% file: tree_prop2_fig.pstex_t
\begin{picture}(0,0)%
\includegraphics{tree_prop2_fig.pstex}%
\end{picture}%
\setlength{\unitlength}{3947sp}%
\begingroup\makeatletter\ifx\SetFigFont\undefined%
\gdef\SetFigFont#1#2#3#4#5{%
  \reset@font\fontsize{#1}{#2pt}%
  \fontfamily{#3}\fontseries{#4}\fontshape{#5}%
  \selectfont}%
\fi\endgroup%
\begin{picture}(6600,2519)(1714,-2534)
\put(7688,-522){\makebox(0,0)[lb]{\smash{{\SetFigFont{12}{14.4}{\rmdefault}{\mddefault}{\updefault}{\color[rgb]{0,.56,0}$\al_i$}%
}}}}
\put(3001,-2461){\makebox(0,0)[lb]{\smash{{\SetFigFont{12}{14.4}{\rmdefault}{\mddefault}{\updefault}{\color[rgb]{0,0,0}$\pi(Y)$}%
}}}}
\put(6999,-2424){\makebox(0,0)[lb]{\smash{{\SetFigFont{12}{14.4}{\rmdefault}{\mddefault}{\updefault}{\color[rgb]{0,0,0}$\pi(X)$}%
}}}}
\put(5885,-2438){\makebox(0,0)[lb]{\smash{{\SetFigFont{12}{14.4}{\rmdefault}{\mddefault}{\updefault}{\color[rgb]{0,0,0}$G(r)$}%
}}}}
\put(3621,-562){\makebox(0,0)[lb]{\smash{{\SetFigFont{12}{14.4}{\rmdefault}{\mddefault}{\updefault}{\color[rgb]{0,0,1}$\beta$}%
}}}}
\put(7107,-901){\makebox(0,0)[lb]{\smash{{\SetFigFont{12}{14.4}{\rmdefault}{\mddefault}{\updefault}{\color[rgb]{0,0,0}$X$}%
}}}}
\put(3057,-904){\makebox(0,0)[lb]{\smash{{\SetFigFont{12}{14.4}{\rmdefault}{\mddefault}{\updefault}{\color[rgb]{0,0,0}$Y$}%
}}}}
\end{picture}%

%% file: BehIn.pstex_t
\begin{picture}(0,0)%
\includegraphics{BehIn.pstex}%
\end{picture}%
\setlength{\unitlength}{3947sp}%
\begingroup\makeatletter\ifx\SetFigFont\undefined%
\gdef\SetFigFont#1#2#3#4#5{%
  \reset@font\fontsize{#1}{#2pt}%
  \fontfamily{#3}\fontseries{#4}\fontshape{#5}%
  \selectfont}%
\fi\endgroup%
\begin{picture}(5652,2568)(-611,-2144)
\put(-224,-216){\makebox(0,0)[lb]{\smash{{\SetFigFont{12}{14.4}{\rmdefault}{\mddefault}{\updefault}{\color[rgb]{0,0,0}$A$}%
}}}}
\put(871,-389){\makebox(0,0)[lb]{\smash{{\SetFigFont{12}{14.4}{\rmdefault}{\mddefault}{\updefault}{\color[rgb]{0,0,0}$y$}%
}}}}
\put(4464,-413){\makebox(0,0)[lb]{\smash{{\SetFigFont{12}{14.4}{\rmdefault}{\mddefault}{\updefault}{\color[rgb]{0,0,0}$C$}%
}}}}
\put(3308,-360){\makebox(0,0)[lb]{\smash{{\SetFigFont{12}{14.4}{\rmdefault}{\mddefault}{\updefault}{\color[rgb]{0,0,0}$z$}%
}}}}
\put(3713,-676){\makebox(0,0)[lb]{\smash{{\SetFigFont{12}{14.4}{\rmdefault}{\mddefault}{\updefault}{\color[rgb]{0,0,0}$x$}%
}}}}
\put(5026,-1936){\makebox(0,0)[lb]{\smash{{\SetFigFont{12}{14.4}{\rmdefault}{\mddefault}{\updefault}{\color[rgb]{0,0,0}$B$}%
}}}}
\put(3788,-2057){\makebox(0,0)[lb]{\smash{{\SetFigFont{12}{14.4}{\rmdefault}{\mddefault}{\updefault}{\color[rgb]{0,0,0}$p$}%
}}}}
\put(1020,-2071){\makebox(0,0)[lb]{\smash{{\SetFigFont{12}{14.4}{\rmdefault}{\mddefault}{\updefault}{\color[rgb]{0,0,0}$q$}%
}}}}
\end{picture}%